\documentclass[12pt]{article}

\usepackage{latexsym,amsfonts,amsmath,epsfig,tabularx,amsthm,dsfont}

\theoremstyle{definition}

\newtheorem{thm}{Theorem}[section]
\newtheorem{rem}[thm]{Remark}
\newtheorem{lem}[thm]{Lemma}
\newtheorem{prop}[thm]{Proposition}
\newtheorem{cor}[thm]{Corollary}

\newcommand\reals{\mathbb{R}}
\newcommand\R{{\mathbb{R}}}
\renewcommand\natural{\mathbb{N}}
\newcommand\nat{\mathbb{N}}
\newcommand\N{\mathbb{N}}
\newcommand\sphere{\mathbb{S}}
\newcommand\eps{\varepsilon}
\newcommand\ga{\gamma}
\newcommand\dist{{\rm dist}}
\newcommand\rad{{\rm rad}}
\newcommand\diam{{\rm diam}}
\newcommand\Lip{{\rm Lip}}
\newcommand\eu{{\rm e}}

\newcommand{\e}{\varepsilon}
\newcommand{\wt}{\widetilde}
\newcommand{\widebar}[1]{\mbox{\kern1.5pt\hbox{\vbox{\hrule height 0.6pt \kern0.35ex
        \hbox{\kern-0.15em \ensuremath{#1 }\kern0.0em}}}}\kern-0.1pt}
\newcommand\il{\left<}
\newcommand\ir{\right>}

\newcommand{\points}{\mathcal{P}}
\newcommand{\E}{\mathbb{E}}
\renewcommand{\P}{\mathbb{P}}
\newcommand{\Pbf}{{\bf P}}
\newcommand{\ind}[1]{\mathds{1}_{#1}}
\newcommand\dint{{\rm d}}

\newcommand{\abs}[1]{\left\vert #1 \right\vert}
\newcommand{\norm}[1]{\left\Vert #1 \right\Vert}

\newlength{\fixboxwidth}
\title{The Curse of Dimensionality for
Numerical Integration of Smooth Functions II}

\author{Aicke Hinrichs,
Erich Novak\footnote{This
author was partially supported by the DFG-Priority Program 1324.}, 
Mario Ullrich\footnote{This author was partially supported by 
DFG GRK 1523.}\\
Mathematisches Institut, Universit\"at Jena\\
Ernst-Abbe-Platz 2, 07743 Jena, Germany\\
email: a.hinrichs@uni-jena.de,
erich.novak@uni-jena.de, \\
ullrich.mario@gmail.com\\
\qquad
\\
Henryk Wo\'zniakowski\footnote{This author was partially
supported by the National Science
Foundation.
}\\
Department of Computer Science, Columbia University,\\
New York, NY 10027, USA, and\\
Institute of Applied Mathematics, University of Warsaw\\
ul. Banacha 2, 02-097 Warszawa, Poland\\
email:\ henryk@cs.columbia.edu}

\begin{document}

\maketitle

\begin{center}
\emph{
Dedicated to J. F. Traub and G. W. Wasilkowski\\
on the occasion of their 80th and 60th birthdays}
\end{center}

\newpage
\begin{abstract}
We prove the curse of dimensionality in the worst case setting
for numerical
integration for a number
of classes of smooth $d$-variate functions.
Roughly speaking, we consider different bounds for the 
directional or partial derivatives 
of $f \in C^k(D_d)$ and ask whether the curse of dimensionality 
holds for the respective classes of functions. 
We always assume that $D_d \subset \R^d$ has volume one
and we often assume additionally 
that $D_d$ is either 
convex or that its radius is proportional to $\sqrt{d}$.  
In particular, $D_d$ can be the unit cube. 
We consider various values of $k$ including the case 
$k=\infty$
which corresponds to infinitely differentiable functions.
We obtain necessary and sufficient conditions,
and in some cases a full characterization for the curse of dimensionality.
For
infinitely differentiable functions
we prove the curse if the bounds on the successive derivatives
are appropriately large. The proof technique is based on a volume estimate of
a neighborhood of
the convex hull of $n$ points which decays exponentially fast in~$d$.
For $k=\infty$, we also study conditions for quasi-polynomial, 
weak and uniform 
weak tractability. In particular, weak tractability holds 
if all directional derivatives are bounded by one. It is still an open problem
if weak tractability holds if all partial derivatives are bounded by one. 
\end{abstract}

\section{Introduction}  \label{sec:Introduction}

We study the problem of numerical integration, i.e., of approximating
the integral
\begin{equation}\label{eq:int}
S_d(f) = \int_{D_d} f(x) \, \dint x
\end{equation}
over an open subset $D_d\subset \R^d$ of 
Lebesgue measure $\lambda_d(D_d)=1$ for integrable functions
$f\colon D_d\to\R$. 
In particular, we consider the case of \emph{smooth} integrands. 
The main interest is on the behavior of the minimal number of function values
that are needed in the worst case setting
to achieve an error at most $\eps>0$, 
while the dimension $d$ tends to infinity.
Note that classical examples of domains $D_d$ are the unit cube $[0,1]^d$ and 
the normalized Euclidean ball (with volume 1),
which are closed. However, we work with their interiors 
for definiteness of certain derivatives. 
Obviously, this does not change the integration problem.

We always consider sets $D_d$ for which  
$\lambda_d(D_d)=1$. This assumption guarantees that the integration problem
is properly normalized and suffices to establish the curse of dimensionality
for a number of classes considered in this paper. 
To obtain necessary and sufficient conditions on the curse, we need 
further assumptions on $D_d$. Typically we assume that $D_d$ is the unit cube or 
that $D_d$ is convex or that $D_d$ satisfies  property $(\Pbf)$ which roughly says that
the radii of $D_d$ are proportional to $\sqrt{d}$. 

For arbitrary sequences $(D_d)_{d\in\N}$, we prove that numerical
integration suffers 
from the \emph{curse of dimensionality} for certain classes of smooth functions 
with suitable bounds on the Lipschitz constants of directional or partial derivatives
that may depend on $d$. 
The curse of dimensionality means that the minimal 
number of function evaluations is exponentially large in $d$. 
The Lipschitz constants are always defined with respect to the Euclidean distance.
This paper is a continuation of our paper~\cite{HNUW12}
with the following new results:

\begin{itemize} 

\item
We provide nontrivial volume estimates, 
see Theorem 2.1 and 2.3.
We prove that the volume of a neighborhood of the convex hull
of $n$ arbitrary points is exponentially small  in $d$.  

\item
We obtain matching lower and upper bounds for Lipschitz functions, 
see Theorem 3.1.
We prove that if the radii of $D_d$ are proportional to $\sqrt{d}$ 
then the curse holds iff 
$\limsup_{d\to\infty}L_{0,d}\, \sqrt{d}>0$, where $L_{0,d}$ is the Lipschitz constant 
of functions.

\item
We obtain matching lower and upper bounds for functions with a 
Lipschitz gradient, see Theorem 4.1.
We prove that if the radii of convex $D_d$ are proportional to $\sqrt{d}$ 
then the curse holds iff 
$\limsup_{d\to\infty}L_{0,d}\, \sqrt{d}>0$ and 
$\limsup_{d\to\infty}L_{1,d}\, d>0$, where $L_{1,d}$ 
is the Lipschitz constant 
of first directional derivatives of functions. 

\item
We provide  lower and upper bounds 
for functions with higher
smoothness $k>1$, see Theorem 5.1. 
Our lower bounds are sometimes better than those presented in~\cite{HNUW12},
whereas the upper bounds are new. Unfortunately, our lower and upper bounds do not 
always match.
We prove that if the radii of $D_d$ are proportional to $\sqrt{d}$ 
then the curse holds if 
$\limsup_{d\to\infty}L_{0,d}\, \sqrt{d}>0$ and 
$\limsup_{d\to\infty}L_{j,d}\,d>0$ for all $j=1,\dots,k$, where $L_{j,d}$ 
is the Lipschitz constant 
of $j$th directional derivatives of functions. 
On the other hand, if $\lim_{d\to\infty}L_{j,d}d^{(j+1)/2}=0$ for some $j\in\{0,1,\dots,k\}$
then the curse does not hold. Hence, our bounds match only if $j\in\{0,1\}$.  

\item
We obtain results for $C^\infty$ functions, see Theorem~\ref{thm5} and \ref{prop:uwt}.
In particular, in this case we also study 
quasi-polynomial,  weak and uniform weak tractability.
Quasi-polynomial tractability means that the logarithm of the minimal number
of function values 
that are needed to guarantee an error $\eps>0$ 
is bounded proportionally to $(1+\ln\,d)(1+\ln\,\e^{-1})$,
whereas weak tractability means that this number of function values 
is not exponential in $d$ and 
$\e^{-1}$, and uniform weak tractability means that it is not exponential in any 
positive power of $d$ and $\e^{-1}$.  
In particular, we prove that weak tractability holds if all
directional derivatives are bounded by one, see Corollary~\ref{cor:ub_infinity_unit_ball}. 
It is known that
strong polynomial tractability does not hold, i.e., 
the minimal number of function values cannot be bounded by a polynomial in $\e^{-1}$
independently of $d$. It is not known 
if, in particular, we have quasi-polynomial tractability in this case.
It is also open if weak tractability holds 
for the larger class of all partial derivatives bounded by one, 
see Open Problem 2 of \cite{NW08}.
\end{itemize}

Technical tools used in this paper include:

\begin{itemize} 

\item 
Bounds for the volume of 
$\{ x \in \R^d \mid \dist(x, K) \le \gamma \}$,
where $K$ is the convex hull of $n$ points,
and $\dist$ is the Euclidean distance of a point $x$ from $K$, 
see Theorem 2.1 and~2.3.

\item
Properties of the convolution derived mainly in~\cite{HNUW12}, see Theorem 2.4.
\end{itemize} 

\section{Preliminaries and Tools} \label{sec:Preliminaries and Tools}

\subsection{Complexity} \label{subsec:Complexity}

In this section we precisely define our problem. Let $F_d$  be a class 
of continuous integrable functions $f:D_d\to\reals$. For $f\in F_d$,
we approximate the integral $S_d(f)$, see~\eqref{eq:int}, by algorithms
$$
A_{n,d}(f)=\phi_{n,d}(f(x_1),f(x_2),\dots,f(x_n)),  
$$
where $x_j\in D_d$ can be chosen adaptively and $\phi_{n,d}:\reals^n\to
\reals$ is an arbitrary mapping. Adaption means that the selection of $x_j$
 may depend on the already computed values $f(x_1),f(x_2),\dots,f(x_{j-1})$.
The (worst case) error of the algorithm $A_{n,d}$ is defined as
$$
e(A_{n,d})=\sup_{f\in F_d}|S_d(f)-A_{n,d}(f)|.
$$
Then the information complexity $n(\eps,F_d)$ is 
the minimal number of function values 
which is needed to guarantee that the error is
at most $\eps$, i.e., 
$$
n(\eps,F_d)=\min\{\,n\ |\ \ \exists\ A_{n,d}\ \ \mbox{such that}\ \ 
e(A_{n,d})\le\eps\}.
$$
Hence, we minimize $n$ over all choices of adaptive sample points $x_j$ and
mappings $\phi_{n,d}$. It is well known that as long as the class $F_d$
is convex and symmetric 
we may restrict the minimization of $n$ by
considering only nonadaptive choices of $x_j$ and  linear mappings $\phi_{n,d}$.
Furthermore, in this case we have 
$$
n(\eps,F_d)=\min\{\,n\ |\ \ \inf_{x_1,x_2,\dots,x_n\in D_d}\ \sup_{f\in F_d,\
f(x_j)=0,\, j=1,2,\dots,n} |S_d(f)|\le \eps\}, 
$$
see e.g., \cite[Lemma 4.3]{NW08}.
In this paper we always consider convex and symmetric $F_d$ 
so that we can use the last formula for $n(\eps,F_d)$. 
It is also well known that for convex and symmetric $F_d$ the total
complexity, i.e., the minimal cost of computing an $\eps$
approximation, 
insignificantly differs from the information
complexity.
For more details see, for instance,  Section 4.2.2 of Chapter 4 in~\cite{NW08}.

By the \emph{curse of dimensionality} we mean that 
$n(\eps,F_d)$ is exponentially large in $d$. 
That is, there are positive numbers $c$, $\eps_0$ and $\gamma$ such that
\begin{equation}\label{curse}
n(\eps,F_d) \ge c \, (1+\gamma)^d\ \ \ \ 
\mbox{for all}\ \ \ \eps \le \e_0\ \  \mbox{and infinitely many}\ \ d\in \natural. 
\end{equation}
There are many classes $F_d$ for which the curse of dimensionality has 
been proved for numerical integration and other multivariate problems,
see~\cite{NW08,NW10} for such examples.
In this paper we continue our work from~\cite{HNUW12}. 

\subsection{Function Classes} \label{subsec:Function Classes}

Already in \cite{HNUW12} we considered classes of functions
with bounds on the Lipschitz constants
of all successive directional derivatives up to some order $r$. 
This will also be one of the
main smoothness assumptions in this paper. 
To make clear why this is a 
natural assumption
we now comment on  
the relation between usual and directional derivatives 
in terms of norms of higher derivatives 
viewed as multilinear functionals.
To this end, let $\Omega \in \big\{ D_d, \R^d\big\}$ and
let $f\colon \Omega\to \R$ be an
$r$-times continuously differentiable function.
We denote the class of $r$-times continuously differentiable functions
on $\Omega$ by $C^r(\Omega)$. 
The corresponding
classes of infinitely differentiable functions 
are similarly denoted for $r=\infty$.

For $k=1,\dots,r$, the $k$-th derivative $f^{(k)}(x)$ at a 
point $x\in \Omega$  is naturally considered
as a symmetric $k$-linear map $f^{(k)}(x)\colon (\R^d)^k \to \R$.
Let $\sphere^{d-1}$ be the unit sphere in $\reals^d$.
For $\theta\in\sphere^{d-1}$, let 
$D^{\theta}f(x)=\lim_{h\to0}\frac1h
\bigl(f(x+h\theta)-f(x)\bigr)$ be the derivative in direction $\theta$.
For example, in the case $k=2$ the second derivative is 
the bilinear map defined by the Hessian. 

For $\theta_1,\dots,\theta_k\in \sphere^{d-1}$, the successive 
directional derivative in the directions
$\theta_1,\dots,\theta_k$ is then given as
$$ D^{\theta_k} \dots D^{\theta_1} f (x) = f^{(k)}(x)(\theta_1,\dots,\theta_k)$$
and is independent of the ordering of the derivatives.

The norm of such a $k$-linear map $A : (\R^d)^k \to \R$
is given as
$$ 
\|A\| = \sup \bigl\{ | A(\theta_1,\dots,\theta_k)| \ 
\big|\ \  \theta_1,\dots,\theta_k\in \sphere^{d-1} \bigr\}.
$$
Since the polarization constant of a Hilbert space equals one, 
see \cite[Proposition 1.44]{D99}, this norm is also equal to
$$ 
\|A\| = \sup \bigl\{ | A(\theta,\dots,\theta)| \ \big|\ \  \theta \in \sphere^{d-1} \bigr\}.
$$

For $k\le r$ and $f\in C^r(\Omega)$, let us denote
$$ 
\Vert f^{(k)} \Vert = \Vert f^{(k)} \Vert_\infty 
= \sup_{x\in \Omega} \Vert f^{(k)} (x) \Vert
$$
and
$$ 
\Lip(f^{(k)}) = \sup_{ x,y\in \Omega, x \neq y} 
\frac{\Vert f^{(k)}(x)-f^{(k)}(y)\Vert}{\Vert x-y\Vert_2}.
$$
Then
$$ 
\Lip(f^{(k)}) = \sup_{\theta_1,\dots,\theta_k\in \sphere^{d-1}}\, 
\Lip(D^{\theta_1}\dots D^{\theta_k} f) 
= \sup_{\theta\in \sphere^{d-1}}\, \Lip(D^{\theta}\dots D^{\theta} f). 
$$
Moreover, for $k<r$ we have
$$ \Vert f^{(k+1)} \Vert = \Lip(f^{(k)}). 
$$
If we need to emphasize the domain $\Omega$ in these notations, we will write
$ \Vert f |_{\Omega} \Vert$ and $\Lip(f|_{\Omega})$.
As usual, $f^{(0)}=f$ in the case $k=0$ with 
$$
\norm{f}_\infty = \sup_{x\in\Omega} |f(x)|.
$$
We will use these facts without further comment.

We now describe the function classes we consider in this paper.
The functions shall be defined on $\Omega$.
To make lower bounds for the information complexity as strong as possible, the 
function class should be as small as
possible. Analogously, to make upper bounds as strong as possible, the function
class should be as large as possible. That is why we use two kinds of function
classes. For lower bounds, we require bounds for the Lipschitz constants of certain
directional derivatives. 

To make this precise, fix an $r \in \N_0:=\{0,1,\dots\}$ 
and a double sequence 
$$L = (L_{j,d})_{j\le r, d\in\N}$$ 
of positive numbers.
Now we define the function classes
\begin{equation}\label{eq:class}
  C_d^{r}(L) \,=\,\bigl\{f\in C^{r}(D_d)\ \ \big|\ \ \ \|f\|_\infty \le1,\ 
        \Lip( f^{(j)}) \le L_{j,d}\ \mbox{for all } j\le r\, \bigr\},
\end{equation}
and
$$
\widebar{C}_d^{r}(L) \,=\,\bigl\{f |_{D_d}\ \big|\ \ \ f\in C^{r}(\R^d),\ 
\|f\|_\infty \le1,\ 
        \Lip( f^{(j)}) \le L_{j,d}\ \mbox{for all } j\le r\, \bigr\}.
$$
Obviously,
$$ \widebar{C}_d^{r}(L) \subset C_d^{r}(L),$$
and usually $\widebar{C}_d^{r}(L)$ is a proper subset of $C_d^{r}(L)$. 
This notation is also for $r=\infty$.  

Although we are mainly interested in results for $C_d^{r}(L)$, 
we will sometimes 
prove lower bounds on 
the information complexity for $\widebar{C}_d^{r}(L)$ which imply the same
lower bounds for~$C_d^{r}(L)$.

\subsection{Convex Hull}\label{subsec:Convex Hull}

As already mentioned in the introduction,
the lower bounds on multivariate integration presented in this paper
are based on a volume estimate of a neighborhood of
certain sets in $\R^d$.
Generally, these sets are convex hulls (or their neighborhoods)
of $n$ points in the set $\Omega\subset\R^d$.
Since we need the $\sqrt{d}$-scaling of the distance throughout this
article, we will omit it in the notation from now on.
For instance, we denote by $A_\delta$ the
$(\delta\sqrt{d})$-neighborhood 
of~$A\subset\R^d$, 
which is defined by
\begin{equation}\label{eq:neighbor}
A_\delta = \{ x \in \R^d \mid \dist(x, A) \le \delta\sqrt{d} \} ,
\end{equation}
where $\dist(x,A)=\inf_{a\in A}\|x-a\|_2$ denotes the Euclidean distance of $x$ from $A$.
We also denote by $B_\delta^d(x)$
the $d$-dimensional Euclidean
ball with center $x\in\R^d$, and radius $\delta\sqrt{d}$, i.e.,
\[
B_\delta^d(x)\,=\, \{y\in\R^d \mid \norm{y-x}_2\le\delta\sqrt{d}\}.
\]
We begin with a result that holds for arbitrary sets $D_d$
as long as their \emph{radius} is small enough.
The radius of a set $D_d\subset\R^d$ is defined by
\[
\rad(D_d) \,=\, \inf_{x\in \R^d} \sup_{y\in D_d}
        \Vert y-x\Vert_2.
\]
We prove the following theorem.

\begin{thm} \label{thm:volume-rad}
Let $D_d\subset\R^d$ be bounded with $\lambda_d(D_d)=1$ 
and let 
$R_d=\frac{\rad(D_d)}{\sqrt{d}}$.
Let $K$ be the convex hull of $n$ points $x_1,\dots,x_n\in D_d$.
Then 
\[
\lambda_d(K_\delta) \,<\, n\, \biggl((R_d + 2\delta) \sqrt{\frac{\pi\eu}{2}} \biggr)^d.
\]
This is exponentially small for $R_d<\sqrt{\frac{2}{\pi \eu}}\approx0.4839$ for large $d$ 
and $\delta<1/\sqrt{2\pi\eu}-\frac12 R_d$.
\end{thm}

Some of our results will be based on this estimate and thus,
for convenience, whenever we refer to \emph{sets $(D_d)$ with small radius}
we simply mean that $\lambda(D_d)=1$ and 
\begin{equation}
\label{smallradius}
 \limsup_{d\to\infty} \frac{\rad(D_d)}{\sqrt{d}} <\sqrt{\frac{2}{\pi \eu}}.
\end{equation} 

Note that, unfortunately, this result does not cover the most natural
case of the unit cube for which
$R_d=1/2$, but the Euclidean  ball of volume 1 is covered.
Because of the 
importance of the unit cube 
we will treat it
separately after the proof of the theorem.

Let us also comment on the case that $D_d^p$ is the ball of volume one in the $\ell_p^d$-norm
for $1\le p<\infty$.
Recall that volume and radius of the unit ball $B_p^d$ of $\ell_p^d$ are given by
$$
 \lambda_d(B_d^p) = \frac{2^d \Gamma\big( 1+\frac{1}{p} \big)^d}{\Gamma\big( 1+\frac{d}{p} \big)}
 \qquad \mbox{and} \qquad
 \rad(B_d^p) = d^{\max\{0,1/2-1/p\}}.
$$
Hence the radius of the volume normalized ball of $\ell_p^d$ is
\begin{equation}
 \label{eq:radlp} 
 \rad(D_d^p) = \frac{\Gamma\big( 1+\frac{d}{p} \big)^{1/d}}{2 \Gamma\big( 1+\frac{1}{p} \big)} 
 \, d^{\max\{0,1/2-1/p\}}. 
\end{equation}
Using Stirling's approximation we obtain that 
$$
 \limsup_{d\to\infty} \frac{\rad(D_d^p)}{\sqrt{d}} = 
 \begin{cases}
   \infty & \  \text{ if }\, 1\le p <2,\\
   \frac{1}{2 (p\,\eu)^{1/p} \Gamma(1+1/p)} & \  \text{ if }\, 2\le p <\infty.\\
\end{cases}
$$
Hence $D_d^p$ is a set of small radius iff $2\le p < p^*$ where $p^*\in (2,\infty)$ is the unique solution of 
$$ 
 2 (p^*\,\eu)^{1/p^*} \Gamma\Bigl(1+\frac{1}{p^*}\Bigr) = \sqrt{\frac{\pi \eu}{2}}.
$$
We checked numerically using Matlab that $p^*\approx170.5186$.

\begin{proof}
Observe that for bounded $D_d$ the infimum in the definition of the radius $\rad(D_d)$
is actually a minimum.
Let $z\in\R^d$ be a point where this minimum is attained.
By the result of Elekes~\cite{E86}, the convex hull $K$ is contained
in the union of the $n$ balls, say $C^{(i)}$, with center $\frac{z+x_i}{2}$,
$i=1,\dots,n$, and radius $(R_d/2)\sqrt{d}$.
Recall that $\lambda_d(B_\delta^d) <  \left( \delta\sqrt{2 \pi \eu}\right)^{d}$,
see e.g.~$(6)$ in \cite{HNUW12}.
This implies that
\[
K_\delta \,\subset\, \bigcup_{i=1}^n C^{(i)}_\delta.
\]
Thus, for any $i$ we have 
\[
\lambda_d(K_\delta) \,\le\, n\, \lambda_d(C^{(i)}_\delta)
\,<\, n\, \Bigl((R_d/2+\delta)\,\sqrt{2 \pi \eu}\Bigr)^d.
\]
\end{proof}

We now turn to the case of the unit cube.
First, we state a lemma that bounds the volume of the
intersection of the cube with a single ball. Additionally, 
we estimate 
a value of the involved constant that will be 
important later.

\begin{lem} \label{lem:volume}
Let $z=(1/2,\dots,1/2)$ be the midpoint of the cube and assume that
$$|x_i-z_i| =\eta\ \ \ \mbox{for}\  i=1,\dots,d.
$$
Then, for every
$\delta<\sqrt{\eta^2+1/12}$, there exists a constant
$\ga=\ga(\delta,\eta)<1$ such that
\[
\lambda_d\bigl(B_\delta^d(x)\cap[0,1]^d\bigr) \,\le\, \ga^d.
\]
In particular, we numerically check that 
$\gamma(1/4+1/100, 1/4) \,<\, \frac{7}{8}$.
\end{lem}

\begin{proof}
The result was proven earlier by 
Dyer, F{\"u}redi, McDiarmid in \cite{DFM90} for 
$\eta=\delta=1/4$. We follow the same arguments.
First note that, by symmetry, it is enough to consider 
$x_i=z_i+\eta$, $i=1,\dots,d$.
Let $B=B_\delta^d(x)$.
By considering an independent random variable $X=(X_1,\dots,X_d)$
uniformly distributed in $[0,1]^d$,
the volume of $B\cap [0,1]^d$ 
can be expressed as
\[\begin{split}
\lambda_d \Bigl(B\cap [0,1]^d \Bigr)
\,&=\, \P \biggl( \norm{X-x}_2 \le \delta\, \sqrt{d} \biggr)
\,=\, \P \biggl( \sum_{j=1}^d (X_j-1/2-\eta)^2 \le \delta^2\, d \biggr)\\
\,&=\, \P\left(\exp\biggl\{ \alpha \biggl( \delta^2 d
        - \sum_{j=1}^d (X_j-1/2-\eta)^2 \biggr) \biggr\} \,\ge\, 1 \right),
\end{split}\]
where $\alpha>0$ is a free parameter.
Define the random variables
\[
 Y_j = \exp\left\{ \alpha \Bigl( \delta^2
                - (X_j-1/2-\eta)^2 \Bigr) \right\},
\]
such that
\[
\lambda_d \Bigl(B\cap [0,1]^d \Bigr)
\,=\, \P\biggl(\prod_{j=1}^d Y_j \,\ge\, 1 \biggr).
\]
We now use Markov's inequality
\[
 \P ( Y \ge 1 ) \,\le\, \E (Y),
\]
which holds for all non-negative random variables $Y$, and obtain
\[
\P\biggl(\prod_{j=1}^d Y_j \ge 1 \biggr)
 \le \E\biggl(\prod_{j=1}^d Y_j\biggr) = \prod_{j=1}^d \E(Y_j).
\]
The last equality follows from the independence of the $Y_j$.
It remains to prove 
$$\E(Y_1)\le\ga(\delta,\eta)<1$$ 
for a suitable choice of $\alpha$.
To this end, observe that
\[
\E(Y_1) \,=\, \int_{0}^{1} \exp\left\{ \alpha \Bigl(
        \delta^2-(1/2+\eta)^2 + 2x(1/2+\eta) - x^2 \Bigr) \right\} \,
        \dint x
\]
is a differentiable convex function in $\alpha$ with value 1 at $\alpha=0$.
Thus, $\E(Y_1)<1$ for small enough $\alpha>0$ if and only if
$\frac{\dint}{\dint\alpha} \E(Y_1)\big|_{\alpha=0}<0$.
We get
\[\begin{split}
\frac{\dint}{\dint\alpha} \E(Y_1)\big|_{\alpha=0}
\,&=\, \int_{0}^{1} \Bigl( \delta^2-(1/2+\eta)^2 +
                2x(1/2+\eta) - x^2 \Bigr) \, \dint x \\
\,&=\,  \delta^2-(1/2+\eta)^2 + (1/2+\eta) - 1/3 \\
\,&=\, \delta^2-\eta^2-1/12,
\end{split}\]
which is less than 0 if $\delta^2<\eta^2+1/12$.
This proves the statement of the lemma.

Note that we can choose
\[
\ga(\delta,\eta) \,=\,
\inf_{\alpha>0}\,  \int_{0}^{1} \exp\left\{ \alpha \Bigl(
        \delta^2-(1/2+\eta)^2 + 2x(1/2+\eta) - x^2 \Bigr) \right\} \,
        \dint x.
\]
The bound on $\ga(1/4+1/100,1/4)$ was computed
numerically by {Geogebra} and Matlab, using $\alpha=9/2$.
\end{proof}

Using this lemma we prove the following volume estimate for a
neighborhood of the convex hull of $n$ points in the cube.

\begin{thm} \label{thm:volume}
Let $K$ be the convex hull of $n$ points
$x_1, x_2, \dots , x_n \in [0,1]^d$.
Then, for every $\delta<\frac1{12}$,
there exists a constant $\wt{\ga}=\wt{\ga}(\delta)<1$ 
(depending only on $\delta$)
such that
\[
\lambda_d \Bigl(K_\delta \cap [0,1]^d\Bigr)
\,\le\, n (d+1)\, \wt{\ga}^d.
\]
In particular, $\wt\ga(1/100)<\frac{7}{8}$.
\end{thm}

\begin{proof}
We closely follow the proof of Theorem 3 in \cite{HNW11},
which  is based on the results from Elekes in \cite{E86} and
Dyer, F{\"u}redi, McDiarmid in \cite{DFM90}.

First, it follows from Carath\'eodory's theorem  that $K\subset[0,1]^d$ is
contained in the convex hull
of at most $k=n(d+1)$ vertices of the unit cube $[0,1]^d$.
So, to prove the claim of the theorem, it is enough to show that
\[
\lambda_d \Bigl(\widebar K_\delta \cap [0,1]^d\Bigr)
\le k\, \wt\ga(\delta)^d,
\]
where $\widebar{K}\subset[0,1]^d$ is the convex hull of $k$ vertices
of $[0,1]^d$.
Note that $\widebar K_\delta$ cannot be a subset of $[0,1]^d$ for
$\delta>0$.

By Elekes' result from~\cite{E86}, $\widebar K$ is contained in the
union of $k$ balls
with radius $\sqrt{d}/4$ and centers in the midpoints of the
segments from the corresponding vertex to
the midpoint of the unit cube, i.e.~the coordinates of the
midpoints of the balls, say $y_{i,j}$, satisfy
   $y_{i,j}\in\{1/4,3/4\}$.
This implies that $\widebar K_\delta$ is contained in the union of the $k$
balls with the same centers and radius $(1/4+\delta)\sqrt{d}$.
That is,
\[
\widebar K_\delta \,\subset\, \bigcup_{i=1}^k C_\delta(y_i),
\]
where $y_i=(y_{i,1},\dots,y_{i,d})$ satisfies 
    $y_{i,j}\in\{1/4,3/4\}$, $i=1,\dots,k$, $j=1,\dots,d$, and
$C_\delta(y_i)=\{x\in\R^d \mid
        \Vert x-y_i\Vert_2\le (1/4+\delta)\sqrt{d}\}$.
Hence,
\[
\lambda_d \Bigl(\widebar K_\delta \cap [0,1]^d\Bigr)
\,\le\, \sum_{i=1}^k
        \lambda_d \Bigl(C_\delta(y_i) \cap [0,1]^d\Bigr).
\]

Since $|y_{i,j}-z_j|=1/4$ for all $i=1,\dots,k$ and $j=1,\dots,d$, 
where $z=(1/2,\dots,1/2)$, 
we can apply Lemma~\ref{lem:volume} to obtain the
result with $\wt\ga(\delta)=\ga(1/4+\delta,1/4)$.
\end{proof}

\subsection{Convolutions}\label{subsec:Convolutions}

In this section we recall a result from \cite{HNUW12} which is the
main ingredient for our proof of the curse of dimensionality for
classes of smooth functions.
Roughly speaking, given an initial function, this result shows that
convolution with a (normalized) indicator function of a ball
preserves certain ``nice'' properties of the initial function, while
increasing the degree of its smoothness by one.

For convenience, throughout this section we study  functions
that are defined on $\R^d$. As an obvious corollary of 
Theorem \ref{thm:conv} below we will obtain that the restrictions 
of the constructed functions to the
unit cube 
satisfy the same bounds.

To be precise, fix a number $\delta>0$, $k\in\N$ and a sequence
$(\alpha_j)_{j=1}^k$ with $\alpha_j>0$  such that
\[
\sum_{j=1}^k \alpha_j \le 1.
\]
For example, we may take
$\alpha_j=1/k$  for $j=1,2\,\dots,k$.
Later we will let $k$ tend to infinity. Then
the sequence
$\alpha_j = c_\eta \cdot j^{-1-\eta}$ with some $\eta>0$
and $c_\eta = \frac1{\zeta (1+\eta)}$
will be our choice.
Here, $\zeta$ denotes the Riemann zeta function.

For $j=1,\dots,k$, we define the ball
$$
B_j \,=\, \Bigl\{x\in\R^d \,\big|\ \
\Vert x\Vert_2\,\le\,\alpha_j\, \delta \sqrt{d}\Bigr\}
$$
and the function $g_j\colon \R^d\to\R$ by
\begin{equation}\label{eq:gk}
g_j(x) \,=\, \frac{\ind{B_j}(x)}{\lambda_d(B_j)}
\,=\, \begin{cases}
1/\lambda_d(B_j) & \  \text{ if }\, x\in B_j,\\
0 & \  \text{ otherwise. }
\end{cases}
\end{equation}
Recall that the convolution of two functions $f$ and $g$ is defined
by
\[
(f\ast g)(x) \,=\, \int_{\R^d} f(x-t)\, g(t) \,\dint t,
\qquad x\in\R^d.
\]
Additionally recall from Section~\ref{subsec:Function Classes} that 
by the Lipschitz constant of $f$ we mean
\[
\Lip(f) = \sup_{x \not= y} \frac{|f(x)-f(y)|}{\Vert x-y\Vert_2}.
\]

\goodbreak

\begin{thm}\label{thm:conv}
For $k\in\N$ and $f\in C^r(\R^d)$,
define
$$
f_k=f\ast g_1\ast\ldots\ast g_k\quad
\mbox{with\ \ $g_j$ from \eqref{eq:gk}}.
$$
For $d\ge2$, let
$\Omega \subset \R^d$  be Lebesgue measurable and let $\Omega_\delta$ be its
neighborhood defined as in~\eqref{eq:neighbor}. Then
\begin{itemize}
\item[$(i)$]
if $f(x)=0$ for all  $x\in \Omega_\delta$ then
$f_k(x)=0$ for all $x\in \Omega$,
\item[$(ii)$] $\Lip(f_k)  \le  \Lip(f)$,
\item[$(iii)$]
if $\int_{\Omega} f(x+t)\,\dint x \,\ge\, \eps$
for all $t\in\R^d$ with $\Vert t\Vert_2\le\delta\sqrt{d}$ then
$\int_{\Omega} f_k(x) \dint x \ge \eps$,
\item[$(iv)$] for all $\ell\le r$ and
                all $\theta_1,\theta_2,\dots,\theta_r\in \sphere^{d-1}$,
\[
\Lip\Bigl(D^{\theta_{\ell}}\,D^{\theta_{\ell-1}}
\dots  D^{\theta_1}f_k\Bigr)
\le \Lip\Bigl(D^{\theta_{\ell}}\,D^{\theta_{\ell-1}}
 \dots  D^{\theta_1}f \Bigr),
\]
\item[$(v)$] $f_k\in C^{r+k}(\R^d)$,
and for all $\ell\le r$, all $j=1,\dots,k$ and
all $\theta_1,\theta_2,\dots,\theta_{\ell+j}\in \sphere^{d-1}$,
\[
\Lip\Bigl(D^{\theta_{\ell+j}}\, D^{\theta_{\ell+j-1}}
\dots  D^{\theta_{1}}f_k \Bigr)
\le \biggl(\prod_{i=1}^j \frac{1}{\delta \alpha_{i}}\biggr) \,
\Lip\Bigl(D^{\theta_{\ell}}\, D^{\theta_{\ell-1}}
\dots  D^{\theta_{1}}f
\Bigr) .
\]
\end{itemize}
In particular, 
\begin{itemize}
\item[$(vi)$] $\Lip(f_k^{(\ell)}) \le \Lip(f^{(\ell)})\ $ for all $\ell\le r$,
\item[$(vii)$] $\Lip(f_k^{(\ell+j)}) \le 
                \biggl(\prod_{i=1}^j \frac{1}{\delta \alpha_{i}}\biggr) \,
                \Lip(f^{(\ell)})$ for all $\ell\le r\ $ and $j=1,\dots,k$.
\end{itemize}

\end{thm}

See \cite{HNUW12} for the proof of $(i)$ through $(v)$.
Properties $(vi)$ and $(vii)$ are consequences of $(iv)$ and $(v)$, respectively.

\section{Lipschitz Functions} \label{sec:Lipschitz Functions}

In this section we consider Lipschitz functions.
The results of Sukharev~\cite{Su79} imply the curse of dimensionality
for multivariate integration for the class
$$
F_d = \{ f \colon [0,1]^d \to \R \ \mid \ \
\Lip(f)\le1 \} .
$$
We prove the curse of dimensionality for smaller classes
of Lipschitz functions. Roughly speaking,
the curse holds iff the Lipschitz constant
in dimension $d$ is of the order $d^{-1/2}$ or larger.

In the notation of Subsection \ref{subsec:Function Classes}, we consider the classes
$$ 
C_d^0(L) =\Bigl\{f\colon D_d \to \R\ \ \big|\ \ \ \|f\|_\infty\le1,\ 
\Lip(f) \le L_{0,d}=:L_d \Bigr\},
$$
where $L=(L_d)$ and the sequences $(D_d)_{d\in\N}$ 
satisfy the following property. 

We say that $(D_d)_{d\in\N}$ 
satisfies Property $(\Pbf)$ if it is a 
sequence of open sets  with $\lambda_d(D_d)=1$  such 
that there exist a sequence $(x^*_d)_{d\in\N}$, $x_d^*\in D_d$  and $R<\infty$
with
\begin{equation*} \tag{\Pbf}
\lim_{d\to\infty} \lambda_d\bigl(\{x\in D_d \mid \|x-x_d^*\|_2\ge R\sqrt{d}\}\bigr) \,=\,0. 
\end{equation*}
In particular, $(\Pbf)$ holds for all sequences with $\rad(D_d)$ of order $\sqrt{d}$
or the sequences $ \big(D_d^p \big)_{d\in\N}$ of $\ell_p^d$-balls with volume 1, $p>0$. 
For a derivation of the last statement see~\cite{SS91}.
There it is shown that
$$
 \lim_{d\to\infty} \lambda_d\bigl( D_d^p \cap t D_d^2 \bigr) \,=\,1
$$
provided that $t>t_p$ for a constant $t_p$ depending only on $p$ which is
equivalent to  Property $(\Pbf)$.

The main result of this section is the following theorem.

\begin{thm} \label{thm2} 
Assume that the sequence $(D_d)_{d\in\N}$ satisfies
$(\Pbf)$.
Then the curse of dimensionality holds for $C_d^0(L)$ if and only if
$$\limsup_{d\to\infty}L_d\, \sqrt{d}>0.$$
\end{thm}

We prove lower and upper bounds separately in the next two subsections. 

\subsection{Lower Bounds}\label{subsec:Lower Bounds 1}

Here we prove the curse for the classes $\widebar{C}_d^0(L)$ of functions 
that are restrictions to $D_d$ of functions on $\R^d$ with 
$\|f\|\le1$ and $\Lip(f) \le L_d$. This implies also the curse for $C_d^0(L)$.  

\begin{prop} \label{prop:lb_lip}
Let $(D_d)_{d\in\N}$ be an arbitrary sequence of open sets with 
$\lambda(D_d)=1$. 
Then the information complexity for the class $\widebar{C}_d^0(L)$ satisfies
\[
n(\eps,\widebar{C}_d^0(L)) 
\ge (1-a\eps) \left(\frac{a L_d\sqrt{d}}{3\sqrt{2\eu\pi}}\right)^d
\qquad
\text{ for all } \eps\in(0,1/a), d\in\N \text{ and } a\ge1.
\]
This implies 
the curse of dimensionality 
for
the class $\widebar{C}_d^0(L)$ if
$$ \limsup_{d\to\infty} L_d \sqrt{d} > 0.$$
\end{prop}

\begin{proof}
Since $\widebar{C}_d^0(L)$ is convex and symmetric,
we know that adaption does not help, see \cite{Ba71}.
Hence, for given $d$ and $n$,
we may assume that the information on $f$ is given by
function values $f(x_1), f(x_2), \dots , f(x_n)$  for some $x_1,\dots,x_n \in D_d$.
Let $K=\{x_1,\dots,x_n\}$ and consider the function
$$
\widebar{f}(x) =  \min\bigl\{1, L_d \, \dist(x, K)\bigr\}  \quad\text{ for all } x\in \R^d.
$$
Since $\dist( \,\cdot\, ,K)$ has Lipschitz constant 1, we have 
$f := \widebar{f}\big|_{D_d}\in \widebar{C}_d^0(L)$.
Let
\[
\points_\delta \,=\, \bigcup_{i=1}^n B_\delta^d(x_i),
\]
where $B_\delta^d(x_i)$ is the ball with
center $x_i$ and radius $\delta\sqrt{d}=\frac{1}{L_d}$.
Note that $f(x)=1$ for all $x\notin\points_\delta$. 
We obtain
\[\begin{split}
\int_{D_d}  f(x) \,\dint x
\,&\ge\, \int_{D_d \setminus\points_{\delta}} f(x) \,\dint x
\,=\, 1- \lambda_d(\points_{\delta}\cap D_d ) \\
\,&\ge\, 1-\lambda_d(\points_{\delta})
\,\ge\, 1- n\lambda_d(B_{\delta}^d).
\end{split}\]
It is shown in \cite[eq.~(6)]{HNUW12} that 
$$ \lambda_d(B_{\delta}^d) < \frac{\Bigl(3\delta\sqrt{2\eu\pi}\Bigr)^d}{\sqrt{\pi d}}.$$
Therefore,
$$ 
\int_{D_d}  f(x) \,\dint x
\,>\, 1- \frac{n \Bigl(3\delta\sqrt{2\eu\pi}\Bigr)^d}{\sqrt{\pi d}}   
\,>\, 1-n \left(\frac{3\sqrt{2\eu\pi}}{L_d\sqrt{d}}\right)^d.
$$
The bound on $n(\eps,\widebar{C}_d^0(L))$ for $a=1$ follows. 

Now assume that $a>1$.
It follows from $\widebar{C}_d^0(a L)\subset a \widebar{C}_d^0(L)$ that 
\begin{equation}\label{eq:compl_scal}
n\bigl(\eps,\widebar{C}_d^0(L)\bigr) 
\,=\, n\bigl(a \eps,a \widebar{C}_d^0(L)\bigr) 
\,\ge\, n\bigl(a \eps, \widebar{C}_d^0(a L)\bigr).
\end{equation}
A simple substitution using the result above for $a=1$ in the above inequality leads to
\[
n(\eps,\widebar{C}_d^0(L))\ge
n\bigl(a \eps,\widebar{C}_d^0(a L)\bigr) \,\ge\, (1-a\eps) 
        \left(\frac{a L_d\sqrt{d}}{3\sqrt{2\eu\pi}}\right)^d.
\]
This implies the curse for $\alpha:=\limsup L_d\sqrt{d} >0$.
Indeed, for any $\eta\in(0,\alpha)$ it is enough to take 
$a>\frac{3\sqrt{2\eu\pi}}{\alpha-\eta}$. Then the lower bound on 
$n(\eps,\widebar{C}_d^0(L))$ is exponentially large in $d$ for $\eps<\eps_0=\frac{1}{a}$.
\end{proof}

Note that Proposition~\ref{prop:lb_lip} leads to a super-exponential  lower 
bound on the information complexity, if $L_d$ decays slower than $d^{-1/2}$.

\subsection{Upper Bounds}    \label{subsec:Upper Bounds 1}

In the last subsection we proved the curse of dimensionality for
function classes with (roughly speaking) Lipschitz constant 
bounded from below by a positive multiple of $1/\sqrt{d}$.
In this subsection we complement this result by proving upper bounds which
are simply based on one point formulas.
That is, assuming~$(\Pbf)$, a single evaluation of the function is enough to 
obtain an arbitrary small error as long as $d$ is large enough and 
$\lim_{d\to\infty}L_d \sqrt{d}=0$.
Recall that our function classes are defined by 
$$ C_d^0(L) 
=\bigl\{f\colon D_d \to \R\ \ \big|\ \ \ \|f\|_\infty\le1,\ \Lip(f) \le L_d \bigr\}.$$

\begin{prop} \label{prop:ub_lip}
Assume that the sequence $(D_d)_{d\in\N}$ satisfies $(\Pbf)$.
Let 
$$
\lim_{d\to\infty}L_d \sqrt{d}=0.
$$
Then the information complexity for the classes $C_d^0(L)$ satisfies 
\[
n(\eps,C_d^0(L))
\,=\, 1
\]
for all $\eps\in(0,1)$ and $d\ge d(\eps)$ large enough.
\end{prop}

\begin{proof}
We use the one point formula
\[
A_{1,d}(f) = f(x^*_d), 
\]
where we choose the $x^*_d$ as provided by Property $(\Pbf)$.
With $R>0$ from Property $(\Pbf)$, we obtain for $f\in C_d^0(L)$ that
\[\begin{split}
|S_{d}(f)-A_{1,d}(f)|&=\left|\int_{D_d}(f(x)-f(x^*_d))\,{\rm d}x\right| \\
&\le \left|\int_{D_d\cap B_R^d(x_d^*)}(f(x)-f(x^*_d))\,{\rm d}x\right| 
                + \left|\int_{D_d\setminus B_R^d(x_d^*)}(f(x)-f(x^*_d))\,{\rm d}x\right| \\ 
&\le L_d\,\int_{D_d\cap B_R^d(x_d^*)}\|x-x^*\|_2\,{\rm d}x 
                + 2\,\lambda_d\bigl(D_d\setminus B_R^d(x_d^*)\bigr) \\
&\le R L_d\sqrt{d}
                + 2\,\lambda_d\bigl(D_d\setminus B_R^d(x_d^*)\bigr).
\end{split}\]
Hence,
\[
e(A_{1,d})\le  R L_d\sqrt{d} + 2\,\lambda_d\bigl(D_d\setminus B_R^d(x_d^*)\bigr),
\]
which tends to zero with $d$ approaching infinity.
This proves that $n(\eps,C^0_d(L))\le1$. To finish the proof we need
to show that $n(\eps,C^0_d(L))>0$. Even for $L_d=0$, the class
$C^0_d(L)$ contains all constant functions $f(x)\equiv c$ for $c\in[-1,1]$.
If $n(\eps,C^0_d(L))=0$ then we can use only constant algorithms
$A_{0,d}\equiv \alpha$, where $\alpha$ is independent of $f$. Taking
$f\equiv 1$
and $f\equiv-1$ we have
$$
e(A_{0,d})\ge \max(1-\alpha|,|-1-\alpha|)\ge 1>\eps.
$$
Hence, $n(\eps,C^0_d(L))$ cannot be zero, and this completes the proof.  
\end{proof}

\section{Functions with Lipschitz Gradients} \label{sec:Functions with Lipschitz Gradients}

In this section we want to strengthen the results from the previous
section by proving the curse of dimension for a smaller class of
functions. In fact, we impose bounds on the Lipschitz constants of
the derivatives of the functions that are of order $1/d$ (instead of
$1/\sqrt{d}$). 

Now we can no longer assume arbitrary domains for the lower bounds. 
The essential geometric property of the sets $D_d$ 
that we need 
is that a sufficiently large neighborhood of the convex hull of $n$ points in $D_d$ 
has very small volume as long as $n$ is not exponentially large in $d$. 
We proved such estimates in Section \ref{subsec:Convex Hull}
for sets with small radius and for the unit cube.
Recall that 
\[ 
C_d^1(L) = \bigl\{f\in C^1(D_d)\ \ \big|\ \ \ \|f\|_\infty\le1,\ \Lip(f) \le L_{0,d}, 
        \ \Lip(\nabla f) \le L_{1,d} \bigr\}.
\]

The main result of this section is the following theorem.

\begin{thm} \label{thm3}
Let $(D_d)_{d\in\N}$ be cubes $(0,1)^d$ or 
convex sets of small radius in the sense of \eqref{smallradius}. 
Then the curse of dimensionality holds for $C_d^1(L)$ if and only if 
$$
\limsup_{d\to\infty}L_{0,d} \, \sqrt{d} > 0 \qquad 
\mbox{and} \qquad \limsup_{d\to\infty}L_{1,d} \, d > 0.
$$
\end{thm}

Again, we prove lower and upper bounds separately in the next two subsections. 
Then Theorem \ref{thm3} is a direct consequence of Propositions \ref{prop:lb_lip}, 
\ref{prop:lb_lipgrad}, \ref{prop:lb_lipgrad2}
and \ref{prop:ub_lipgrad_convex}.

\subsection{Lower Bounds}\label{subsec:Lower Bounds 2}

Assume that $(D_d)$ is a sequence of convex sets  of small radius or the unit cube.
To prove the curse of dimensionality for the class $C_d^1(L)$, 
we proceed as in Section~\ref{subsec:Lower Bounds 1} and 
show the curse for the smaller class $\widebar{C}_d^1(L)$. 
For this we construct, for given sample points $x_1,\dots,x_n\in D_d$, 
a fooling function that is defined on the entire 
$\R^d$ and fulfills the required bounds on the Lipschitz constants. 
This function will be zero at the points $x_1,\dots,x_n$ 
and will have a large integral in $D_d$ as long as $n$ is not 
exponentially large in $d$. 
Moreover, this function will be
zero in a neighborhood of the entire convex hull of 
$x_1,\dots,x_n$.
Unfortunately, we were not able to construct a fooling function that is 
zero only at $x_1,\dots,x_n$ (as for Lipschitz functions). 
Such a function would probably be an important step towards an 
improvement of the results of this paper
for the case of higher smoothness.

Let $\phi\colon\R^d \to \R$ be the squared distance function,
which is defined as
\[
\phi(x) = \dist(x,K_\delta)^2,
\]
where $K$ is the convex hull of the $n$ points
$x_1,\dots,x_n\in D_d$ and $K_\delta$ is defined as
in \eqref{eq:neighbor}.
The function $\phi$ 
obviously vanishes on $K_\delta$.
Let $P_{K_\delta}:\R^d \to K_\delta$ be the nearest neighbor projection,
i.e., $P_{K_\delta}(x)$ is the unique point in ${K_\delta}$ given by
\[
\Vert x-P_{K_\delta}(x)\Vert_2 = \dist\big(x,{K_\delta}\big).
\]
It follows from Theorem 3.3 in \cite{DZ94}
that $\phi$ is differentiable with gradient
\[
 \nabla \phi (x) \,=\, 2 \big(x-P_{K_\delta}(x)\big).
\]
Since $P_{K_\delta}$ is a contraction, this also implies
\[
 \Vert \nabla\phi(x) -  \nabla\phi(y) \Vert_2 \le 4 \Vert x-y\Vert_2
\qquad\mbox{for all}\quad x,y\in\R^d.
\]
That is, $\nabla\phi$ is Lipschitz with constant
${\rm Lip}(\nabla\phi)\le 4$.

The fooling function will now be the restriction of a function of the form
\[
 f = p \circ \phi
\]
with $\phi$
as above and with some bounded and smooth function
$p\colon\R_+ \to \R_+$.
Before we state our choice of $p$ explicitly, we now show how (and which)
properties of $p$ imply the needed properties of $f$.
First, if we assume that $p(0)=0$,
we obtain $f(x)=0$ for all $x\in K_\delta$.
Moreover, we have
\[
 \Vert f \Vert_\infty \le \Vert p \Vert_\infty:=\max_{t\in\R}|p(t)|,
\]
so bounds on function values of $p$ directly translate
into bounds on function values of $f$.

Assuming differentiability of $p$, we obtain  the formula
\begin{equation} \label{eq:nabla}
 \nabla f (x) \,=\, p'\big(\phi(x)\big) \nabla \phi(x) \,=\,
2 p'\big(\dist(x,K_\delta)^2\big) \big(x-P_{K_\delta}(x) \big).
\end{equation}
This implies
\[
 \Vert \nabla f (x) \Vert_2
 \,=\, 2 p' \big(\dist(x,K_\delta)^2\big)\,\dist(x,K_\delta).
\]
So any uniform upper bound on the function $2 \sqrt{t}\, p'(t)$ is also an upper
bound for $ \Vert \nabla f (x) \Vert_2$ for all $x\in\R^d$.
Note that
\[
|D^{\theta} f(x)| \,=\, |\il\theta, \nabla f(x)\ir|
\,\le\, \Vert \nabla f (x) \Vert_2 \qquad \text{ for all } \theta\in
\sphere^{d-1}.
\]
This gives an upper bound on 
all directional derivatives of $f$ of order one,  and thus
a bound on the Lipschitz constant of $f$.
Additionally, 
for all $\theta\in \sphere^{d-1}$ we have
\[\begin{split}
\Lip(D^{\theta} f) \,&=\,
\sup_{x,y\in\R^d} \frac{|D^{\theta} f(x)- D^{\theta} f(y)|}{\Vert x-y\Vert_2}
\,=\,
\sup_{x,y\in\R^d} \frac{|\il \theta, \nabla f(x)- \nabla f(y)\ir|}{\Vert x-y\Vert_2}\\
\,&\le\,
\sup_{x,y\in\R^d} \frac{\Vert\nabla f(x)- \nabla f(y)\Vert_2}{\Vert x-y\Vert_2}
\,=\, \Lip(\nabla f),
\end{split}\]
which implies a bound on the Lipschitz constant of the (first-order) derivatives of $f$
given a bound on $\Lip(\nabla f)$.
But, if we assume $p(t)=1$ for $t>\delta^2 d$, then $\nabla f(x)=0$
for all $x\in\R^d$ with $\dist(x,{K_\delta})>\delta\sqrt{d}$,
i.e.,~for all $x\notin K_{2\delta}$. This gives
\[
\Lip(\nabla f) \,=\, \Lip(\nabla f |_{K_{2\delta}}).
\]
Using \eqref{eq:nabla} we obtain
\begin{eqnarray*}
\tfrac12\bigl(\nabla f(x)-\nabla f(y)\bigr)&=&
p^\prime\bigl(\dist(x,{K_\delta})^2\bigr)\bigl(x-P_{K_\delta}(x)\bigr)-
p^\prime\bigl(\dist(y,{K_\delta})^2\bigr)\bigl(y-P_{K_\delta}(y)\bigr)\\
&=&
p^\prime\bigl(\dist(y,{K_\delta})^2\bigr)\Bigl(x-P_{K_\delta}(x) -y+P_{K_\delta}(y)\Bigr)\\
&&\ + \,\Bigl(p^\prime\bigl(\dist(x,{K_\delta})^2\bigr)-p^\prime(\dist(y,{K_\delta})^2)\Bigr)
                                \bigl(x-P_{K_\delta}(x)\bigr).
\end{eqnarray*}
The norm of the first term can be bounded
by $2\|p^\prime\|_\infty\,\|x-y\|_2$,
whereas, for $x,y\in K_{2\delta}$, the second term is bounded by
\[\begin{split}
\Lip(p^\prime)\,\Bigl|\dist(x,K_\delta)^2&-\dist(y,K_\delta)^2\Bigr|\,
        \dist(x,K_\delta) \\
&\le\,\Lip(p^\prime)\,\|x-y\|_2\,\left(\dist(x,K_\delta)+\dist(y,K_\delta)\right)
\,\dist(x,K_\delta)\\
&\le\,
2 \delta^2 d\,\Lip(p^\prime)\,\|x-y\|_2.
\end{split}\]
Here, we used $\dist(x,K_\delta)\le\delta\sqrt{d}$ for
$x\in K_{2\delta}$.
This yields that
\[
\Lip(D^{\theta} f) \,\le\, \Lip(\nabla f)
\,\le\, 4\,\big(\Vert p'\Vert_\infty + \delta^2 d \, \Lip(p') \big).
\]

In summary, we want to construct a differentiable function
$p\colon \R_+\to\R_+$ such that
\begin{itemize}
\item[$(a)$] \quad $p(0)\,=\,0$,
\item[$(b)$] \quad $\norm{p}_\infty=1$,
\item[$(c)$] \quad $p(t)=1\ $ for $t>\delta^2 d$,
\item[$(d)$] \quad $2\sqrt{t}p'(t) \le \frac{2}{\delta\sqrt{d}}\ $ for all $t\in\R$ and
\item[$(e)$] \quad $4\,\big(\Vert p'\Vert_\infty + \delta^2 d \, \Lip(p') \big)
                                                        \,\le\, \frac{40}{\delta^2 d}$.
\end{itemize}
These properties, once verified, imply that
$f|_{D_d}\in \widebar{C}_d^1(L)$ with $f=p\circ \phi$, where
\[
L_{0,d} = \frac{2}{\delta \sqrt{d}}  \quad \text{and} \quad
L_{1,d}=\frac{40}{\delta^2 d}.
\]

We now give an explicit construction of such a function $p$.
We use the function
\[
p(t) \,=\, \begin{cases}
\frac{2t}{\delta^2 d}, & \text{ if }\quad  t\le \frac{\delta^2 d}{4}, \\
-\frac{2t}{\delta^2 d}+\frac{4\sqrt{t}}{\delta \sqrt{d}}-1, & \text{ if }\quad  
        t\in \bigl(\frac{\delta^2 d}{4},\delta^2 d \bigr), \\
1, & \text{ if }\quad  t\ge \delta^2 d.
\end{cases}
\]
We obtain immediately the properties $(a)$--$(c)$ for $p$. 
Furthermore, $p$ is continuously differentiable with derivative 
\[
p'(t) \,=\, \begin{cases}
\frac{2}{\delta^2 d}, & \text{ if }\quad  t\le \frac{\delta^2 d}{4}, \\
-\frac{2}{\delta^2 d}+\frac{2}{\delta \sqrt{d}}\frac{1}{\sqrt{t}}, & 
        \text{ if }\quad  t\in \bigl(\frac{\delta^2 d}{4},\delta^2 d \bigr), \\
0, & \text{ if }\quad  t\ge \delta^2 d.
\end{cases}
\]
Using this, we obtain property $(d)$. 
For $(e)$ observe that $p'$ is absolutely continuous and thus, almost everywhere 
differentiable with derivative
\[
p''(t) \,=\, \begin{cases}
-\frac{1}{\delta \sqrt{d}}\frac{1}{t^{3/2}}, & 
        \text{ if }\quad  t\in \bigl(\frac{\delta^2 d}{4},\delta^2 d \bigr), \\
0, & \text{ otherwise, }
\end{cases}
\]
for all $t\notin\bigl\{\frac{\delta^2 d}{4},\delta^2 d\bigr\}$.
It is well known that in such cases the Lipschitz constant of $p'$ is equal to 
the supremum of $\abs{p''(t)}$ over 
$t\notin\bigl\{\frac{\delta^2 d}{4},\delta^2 d\bigr\}$.
This proves that $\Lip(p')\le8/(\delta^4 d^2)$ and thus,
\[\begin{split}
4\,\big(\Vert p'\Vert_\infty + \delta^2 d \, \Lip(p') \big)
\,&\le\, 4\,\biggl(\frac{2}{\delta^2 d} + \delta^2 d \, 
        \Bigl( \frac{1}{\delta\sqrt{d}}\, \frac{1}{(\delta\sqrt{d}/2)^3} \Bigr) \biggr)\\
\,&\le\, 4\,\biggl(\frac{2}{\delta^2 d} + \frac{8 \delta^2 d}{\delta^4 d^2} \biggr) 
\,=\, \frac{40}{\delta^2 d}.
\end{split}\]
This proves the last property $(e)$.

The following result shows the curse for a specific choice of 
the bounds $L^*=(L^*_{0,d}, L^*_{1,d})_{d\in\N}$ and thus, 
for every $L$ that is a constant multiple of $L^*$.
If we consider the domains to be unit cubes, we obtain

\begin{prop}\label{prop:lb_lipgrad}
Let $D_d=(0,1)^d$. For
\[
L^*_{0,d} \,=\, \frac{400}{\sqrt{d}}  \quad \text{ and } \quad
L^*_{1,d} \,=\, \frac{16\cdot 10^5}{d}
\]
we have
\[
  n\bigl(\eps,\widebar{C}_d^1(L^*)\bigr) \,\ge\, 
  \frac{1-\eps}{d+1}  \left(\frac{8}{7}\right)^{d} \quad
\text{ for all }\; d\in\nat \;\text{ and }\;  \eps\in(0,1).
\]
Hence, the curse of dimensionality 
holds also for the class $\widebar{C}_d^1(L)$ 
if
\[
\limsup_{d\to\infty}L_{0,d} \, \sqrt{d} > 0 
\qquad \mbox{and} \qquad \limsup_{d\to\infty}L_{1,d} \, d > 0.
\] 
\end{prop}

\begin{proof} 
Let $\delta=1/200$.
As discussed above, the restriction $\wt f=f|_{D_d}$ satisfies 
$\wt f\in \widebar{C}_d^1(L^*)$.
It remains to bound its integral in $D_d$. 
By property $(c)$ of $p$ we have 
$\wt f(x)=1$ for all $x$ with $\dist(x,K_\delta)^2>\delta^2 d$, 
i.e.~for all $x\notin K_{2\delta}$. 
Then Theorem~\ref{thm:volume} implies 
\[
\int_{(0,1)^d} \wt f(x) \,\dint x 
\,\ge\, \int_{(0,1)^d\setminus K_{2\delta}} \wt f(x) \,\dint x 
\,=\, 1- \lambda_d(K_{2\delta}\cap [0,1]^d) > 1-n(d+1)\left(\frac{7}{8}\right)^d.
\] 
This proves the curse for $\widebar{C}_d^1(L^*)$ 
if $\limsup L_{0,d} \, \sqrt{d} > 400$ and 
$\limsup L_{1,d} \, d > 2^{21}$. 
Similar arguments as in the proof of Proposition~\ref{prop:lb_lip} 
(using $\widebar{C}_d^1(a L)\subset a \widebar{C}_d^1(L)$ for $a\ge1$) 
conclude the proof in the general case.
\end{proof}

We now consider the class of domains with small radius.
The proof follows exactly the same lines but with the use 
of Theorem~\ref{thm:volume} replaced by Theorem~\ref{thm:volume-rad}. 
Furthermore, in contrast to Theorem~\ref{thm3}, we do not have 
to assume now convexity of the domains $D_d$.

\begin{prop}\label{prop:lb_lipgrad2}
Let $(D_d)_{d\in\N}$ be a sequence of sets with small radius
in the sense of \eqref{smallradius}.
Then the curse of dimensionality holds for the class $\widebar{C}_d^1(L)$, 
if
\[
\limsup_{d\to\infty}L_{0,d} \, \sqrt{d} > 0 
\qquad \mbox{and} \qquad \limsup_{d\to\infty}L_{1,d} \, d > 0.
\] 
\end{prop}

\begin{proof} 
We present only a sketch of the proof, since it is almost identical 
to the proof of Proposition~\ref{prop:lb_lipgrad}.
First, note that for every set of small radius there exists a $\delta>0$ 
such that $\lambda_d(K_{2\delta})$ is exponentially small in $d$, 
where $K$ is the convex hull of $n$ points in $D_d$, 
see Theorem~\ref{thm:volume-rad}.
Choosing this $\delta$ in the above construction of the fooling 
function $f$ (and its restriction $\wt f$) shows 
the curse for the classes $C_d^1(L)$ with 
$\limsup L_{0,d} \, \sqrt{d} > C$ and 
$\limsup L_{1,d} \, d > C$ for some $C<\infty$. 
Again, we obtain the result by scaling.
\end{proof}

\begin{rem}
Note that the calculations of this section could be done also with 
the function $\dist(x,K)^2$ which vanishes only on $K$ instead of 
$K_\delta$. 
This would have somewhat reduced the constants in Proposition \ref{prop:lb_lipgrad}.
For convenience we worked with $\phi$ as above, since we need this function also in 
Section~\ref{sec:Functions with Higher Smoothness}.
\end{rem}

\begin{rem}
With similar ideas as used in \cite{HNUW12}, 
it is not possible to produce super-exponential lower bounds on the 
information complexity in the cases of Proposition~\ref{prop:lb_lipgrad} 
and \ref{prop:lb_lipgrad2}. 
The reason is that even with very small $\delta$ (depending on $d$) the 
volume of the $\delta\sqrt{d}$-neighborhood of the convex hull cannot be 
super-exponentially small. 
However, we conjecture that the information complexity is super-exponential
for slightly larger classes, e.g., if the conditions
\[
\limsup_{d\to\infty}L_{0,d} \, \sqrt{d} > 0 
\qquad \mbox{and} \qquad \limsup_{d\to\infty}L_{1,d} \, d > 0
\] 
are replaced by 
\[
\limsup_{d\to\infty}L_{0,d} \, \sqrt{d} = \infty 
\qquad \mbox{and} \qquad \limsup_{d\to\infty}L_{1,d} \, d = \infty.
\] 
\end{rem}

\begin{rem}
Note that instead of this rather complicated function $p$ it would be possible 
to work with the function $\wt{p}(t)=t/(\delta^2 d)$ to obtain essentially the 
same result. But in this case one would have to do the analysis directly 
for the functions restricted to the subsets $D_d\subset\R^d$ and, 
in addition, one would obtain the desired lower 
bound on the information complexity
only for sufficiently small $\eps$ and not for all $\eps<1$ as above.
\end{rem}

\subsection{Upper Bounds}   \label{subsec:Upper Bounds 2}

In the last subsection we proved  the curse of dimensionality for
function classes with (roughly speaking) 
Lipschitz constant of the gradient bounded 
from below by a positive multiple of~$1/d$.
In this subsection we complement this result 
by proving matching upper bounds.
Again using a one point formula is enough to ensure an arbitrary
small error as long as $d$ 
is large enough and $\lim_{d\to\infty} L_{1,d} \, d=0$.

We present two versions of this result. 
One that holds for convex sets with Property 
$(\Pbf)$ and one that holds for arbitrary convex domains.

Again, we want to deal with a function class as large as possible. 
Therefore, we drop the bounds $\|f\|_\infty\le1$ and 
$\Lip(f) \le L_{0,d}$ and consider 
\[
F_d^1 = \bigl\{ f \colon  D_d \to \R \ \mid \ \ \Lip(\nabla f) \le L_{1,d} \bigr\} 
\,\supset\, C_d^1(L).
\]

\begin{prop} \label{prop:ub_lipgrad_convex}
Let $(D_d)_{d\in\N}$ be convex sets with $\lambda(D_d)=1$.
Then the information complexity for the classes $C_d^1(L)$ and $F_d^1$ 
satisfy 
$$
n(\eps,C_d^1(L))=n(\eps,F_d^1)=1
$$
provided that $L_{1,d}\;\diam(D_d)^2\le\eps$, 
where $\diam(D_d)$ is the diameter of the set $D_d$.
Hence, if
\[
\lim_{d\to\infty}L_{1,d}\;\diam(D_d)^2=0, 
\]
then
\[
n(\eps,C_d^1(L))=n(\eps,F_d^1) \,=\, 1
\]
for all $\eps\in(0,1)$ if $d\ge d(\eps)$ is large enough.
\end{prop}

\begin{proof}
We use the one point formula
$A_{1,d}(f) = f(z)$ where $z\in D_d$ is the centroid (center of gravity)  of $D_d$.
Then we have
$$
  S_d(f) - A_{1,d}(f)  =  \int_{D_d} \big( f(x)-f(z) \big) \, {\rm d} x
= \int_{D_d} r(x) \, {\rm d} x
$$
with
$$
 r(x) = f(x)-f(z) - \nabla f (z) \cdot (x-z)
$$
since the integral over $D_d$ of the function $a\cdot(x-z)$
vanishes for any $a\in\R^{d}$.
We estimate $r(x)$ by using the mean value theorem,
which implies the existence of a point~$y$ on the segment $[x,z]$ such that
$$
 f(x)-f(z) = \nabla f (y) \cdot (x-z).
$$
From the Cauchy-Schwarz inequality we conclude that for $f\in F_d^1$ we have
\begin{eqnarray*}
 |r(x)| &=&   \big| \big( \nabla f (y) - \nabla f (z) \big)
\cdot (x-z) \big|
        \le \Vert \nabla f (y) - \nabla f (z)\, \Vert_2
\Vert x-z \Vert_2 \\
        &\le& L_{1,d} \,\Vert y-z \Vert_2 \,\Vert x-z \Vert_2
        \le L_{1,d} \, \Vert x-z \Vert_2^2
        \le L_{1,d}\, \diam(D_d)^2.
\end{eqnarray*}
Hence, the error of $A_{1,d}$ on $F_{d}^1$ is 
bounded by $L_{1,d}\, \diam(D_d)^2$
and tends to zero for $d\to\infty$. Therefore, the curse of 
dimensionality is
not present since
$n(\eps, F_{d}^1) \le1$ 
for any $\eps >0$ and large enough $d$. 
Repeating the argument used in the proof of 
Proposition~\ref{prop:ub_lip} 
we conclude that $n(\eps, F_{d}^1)=1$, as claimed.  
\end{proof}

Note that Proposition~\ref{prop:ub_lipgrad_convex} is already enough to prove 
a part of the necessary conditions in Theorem~\ref{thm3}.
Namely, it is enough to conclude that $\lim_{d\to\infty}L_{1,d}=0$ implies that the curse does not hold.

Proposition~\ref{prop:ub_lipgrad_convex}  does not hold if $D_d$'s are $\ell_1^d$ balls which 
satisfy $(\Pbf)$. Indeed, in this case,
Proposition~\ref{prop:ub_lipgrad_convex} 
would require that $L_{1,d}\,d^2\to0$ instead
of $L_{1,d}\,d\to0$.
However, the next proposition shows that, if  the sequence $(D_d)_{d\in\N}$ 
satisfies $(\Pbf)$, then the curse does not hold if $L_{1,d}\,d$ tends to zero. 

Unfortunately, we cannot omit the bound on the supremum of $f$ as for the class $F^1_d$.
Therefore, we consider the classes
\[
F_d^2 = \bigl\{ f \colon  D_d \to \R \ 
\mid \ \|f\|_\infty\le1,\ \Lip(\nabla f) \le L_{1,d} \bigr\} 
\,\supset\, C_d^1(L).
\]

\begin{prop} \label{prop:ub_lipgrad_P}
Let $(D_d)_{d\in\N}$ be a sequence of convex sets that satisfies $(\Pbf)$.
Additionally assume 
\[
\lim_{d\to\infty}L_{1,d}\,d=0.
\]
Then the information complexity for the classes $C_d^1(L)$ and $F_d^2$ 
satisfy 
\[
n(\eps,C_d^1(L))=n(\eps,F_d^2) \,=\, 1
\]
for all $\eps\in(0,1)$ if $d\ge d(\eps)$ is large enough.
\end{prop}

\begin{proof}
We use the same techniques as in the proofs of
Propositions~\ref{prop:ub_lip} and \ref{prop:ub_lipgrad_convex}.
Now, the algorithm is the one point formula
$A_{1,d}(f) = f(z)$ where $z\in D_d$ is the centroid of $D_d\cap B_R^d(x^*_d)$, 
where $R$ and $x^*_d$ are from $(\Pbf)$.
Set $B=B_R^d(x^*_d)$. 
Then we have
\[
S_d(f) - A_{1,d}(f)  =  \int_{D_d} \big( f(x)-f(z) \big) \, {\rm d} x
= \int_{D_d\cap B} \wt r(x) \, {\rm d} x 
+ \int_{D_d\setminus B} \big( f(x)-f(z) \big) \, {\rm d} x
\]
with
$$
 \wt r(x) = f(x)-f(z) - \nabla f (z) \cdot (x-z)
$$
since the integral over $D_d\cap B$ of the function $a\cdot(x-z)$
vanishes for any $a\in\R^{d}$.
We estimate $\wt r(x)$ in the same way as in the proof of 
Proposition~\ref{prop:ub_lipgrad_convex} (using $\diam(D_d\cap B)\le R\sqrt{d}$) 
and the second term as in Proposition~\ref{prop:ub_lip} (using $\|f\|_\infty\le1$).
We obtain
\[
|S_d(f) - A_{1,d}(f)| \le R^2 L_{1,d}\, d + 2\,\lambda_d\bigl(D_d\setminus B\bigr)
\]
for all $f\in F^2_d$, 
which, under the assumptions 
of the proposition, tends to zero for $d\to\infty$. 
The rest is as before.
\end{proof}

\section{Functions with Higher Smoothness} \label{sec:Functions with Higher Smoothness}

In this section we deal with the general classes 
$$ C_d^{k}(L) =\{f\in C^{k}(D_d)\ \ \big|\ \ \ \|f\|_\infty\le1,\ 
        \Lip( f^{(j)}) \le L_{j,d}\ \mbox{for } j=0,1,\dots,k\, \}.$$
For $k>1$, our lower and upper bounds will not match anymore even if $D_d=(0,1)^d$.
The upper bound is proved using Taylor's formula which leads to an additional factor 
$1/\sqrt{d}$ for each additional derivative. 
In the proof of the lower bound we will use the smoothing by convolution which 
does not give any additional gain in the
bounds for the higher derivatives. We are stuck with $1/d$ starting from $r=1$.

The main result of this section is the following.

\begin{thm} \label{thm4}
Let $(D_d)_{d\in\N}$ be cubes $(0,1)^d$ or 
sets of small radius in the sense of \eqref{smallradius}. 
For all $k\in\N$, the conditions
\[
\limsup_{d\to\infty}L_{0,d} \, \sqrt{d} > 0 \qquad \mbox{and} \qquad 
\limsup_{d\to\infty}L_{j,d} \, d > 0 \ \mbox{ for }\ j=1,\dots,k
\]
imply the curse of dimensionality for $C_d^k(L)$.

On the other hand, if $(D_d)_{d\in\N}$ is a sequence of convex sets 
with $\lambda(D_d)=1$ and 
there exists $j\in\{0,1,\dots,k\}$ such that
$$ \lim_{d\to\infty}L_{j,d} \, d^{\frac{j+1}{2}} = 0 $$
then the curse does not hold.
\end{thm}

\subsection{Lower Bounds}\label{subsec:Lower Bounds 3}

The lower bounds on the information complexity in this case are 
mainly based on Theorem~\ref{thm:conv}, which shows that 
convolution with certain indicator functions, see \eqref{eq:gk}, 
increases the smoothness of the initial function by loosing 
only a factor
in the bounds on the Lipschitz constants of 
the higher order derivatives. 

For this, recall from Section~\ref{subsec:Convolutions} that we have 
fixed a number $\delta>0$, $\ell\in\N$ 
and a sequence $(\alpha_j)_{j=1}^\ell$ with $\alpha_j>0$  such that
\[
\sum_{j=1}^\ell \alpha_j \le 1.
\]
For the purpose of this section we choose $\alpha_j=1/\ell$ for $j=1,\dots,\ell$.

Additionally, recall that $g_j$ is the normalized indicator function 
of a ball with radius $\alpha_j \delta \sqrt{d}$ and that we define
\[
f_\ell = f\ast g_1\ast\cdots\ast g_\ell.
\]

Let $x_1,\dots,x_n\in D_d$ be the sampling points and $K$ their convex hull. 
The initial function for the convolution will be the function $f\colon \R^d\to\R$
that was constructed in Section~\ref{subsec:Lower Bounds 2}.
This function satisfies the following properties:
\begin{itemize}
        \item $f\in C^1$, 
        \item $f(x)=0$ for $x\in K_\delta$,
        \item $f(x)=1$ for $x\notin K_{2\delta}$,
        \item $\Lip(f) \le \frac{2}{\delta \sqrt{d}}$ and 
        \item $\Lip(f^{(1)}) \le \frac{40}{\delta^2 d}$.
\end{itemize}
By Theorem~\ref{thm:conv} we immediately obtain 
\begin{itemize}
        \item $f_\ell\in C^{\ell+1}$, 
        \item $f_\ell(x)=0$ for $x\in K$,
        \item $f_\ell(x)=1$ for $x\notin K_{3\delta}$,
        \item $\Lip(f_\ell) \le \frac{2}{\delta \sqrt{d}}$ and 
        \item $L^*_{j,d} := \Lip(f_\ell^{(j)}) \le \frac{40}{\delta^2 d} 
                \,\left(\frac{\ell}{\delta}\right)^{j-1}$, 
                                $j=1,2,\dots,\ell+1$.
\end{itemize}
Thus, setting $\ell=k-1$, we have
$f_\ell|_{D_d}\in \widebar{C}_d^k(L^*)\subset C_d^k(L^*)$, where
\begin{equation}\label{eq:L_star}
L^*_{0,d} = \frac{2}{\delta \sqrt{d}} \quad\text{ and }\quad 
L^*_{j,d} = \frac{40}{\delta^2 d} \,
\left(\frac{k-1}{\delta}\right)^{j-1}\ \ \ \mbox{for}\ \  
j=1,\dots,k.
\end{equation}
By the third property of $f_\ell$ we additionally obtain
\[
\int_{D_d} f_\ell(x) \,\dint x 
\,\ge\, \int_{D_d\setminus K_{3\delta}} f_\ell(x) \,\dint x 
\,=\, 1- \lambda_d(K_{3\delta}\cap D_d).
\]

Using the upper bounds for the volume on the right hand side from 
Section~\ref{subsec:Convex Hull}, which were already used in the proofs 
of Proposition~\ref{prop:lb_lipgrad} and \ref{prop:lb_lipgrad2}, 
we obtain the desired lower bounds on the information complexity.
In particular, we obtain for small enough positive $\delta$
(depending only on the sequence $(D_d)_{d\in\N}$), 
that there exists $\eta=\eta(\delta)>1$ 
such that  
\begin{equation}\label{eq:n_higher}
n(\eps,C_d^k(L^*)) \,\ge\, (1-\eps)\, \eta^d
\end{equation}
for all $k\in\N$, $\eps\in(0,1)$ and infinitely many $d\in\N$, 
where the sequence $L^*$ is given by~\eqref{eq:L_star}.

Using the same scaling technique that was used in 
Sections~\ref{subsec:Lower Bounds 1} and \ref{subsec:Lower Bounds 2} 
we obtain the curse of dimensionality
under the assumptions of Theorem~\ref{thm4}.

\begin{rem}
If $D_d$ is the cube we could also give an explicit lower bound 
on the information complexity 
for the class $C_d^k(L^*)$ with $L^*=(L^*_{j,d})$ from above. 
In fact, the same lower bound as in Proposition~\ref{prop:lb_lipgrad} 
holds for every $k\in\N$ if we set $\delta=1/300$.
\end{rem}

\subsection{Upper Bounds}    \label{subsec:Upper Bounds 3}

We now prove that the curse of dimensionality does not hold if the condition 
$$ \lim_{d\to\infty}L_{j,d} \, d^{\frac{j+1}{2}} = 0 $$
holds for some $j\in\{0,1,\dots,k\}$. We first observe that the cases 
$j=0$ and $j=1$ were already dealt with in the previous sections.
For $j\in\{2,3,\dots,k\}$, we use the next proposition.
Note that we prove the result under the 
assumption that the 
sets satisfy Property $(\Pbf)$. Both, the cube and sets of small radius, 
satisfy this assumption.

\begin{prop} \label{prop:ub_higher}
Let $(D_d)_{d\in\N}$ be a sequence of convex sets that satisfies $(\Pbf)$.
Additionally assume that there exists $j\in\{2,3,\dots,k\}$ such that
\[
\lim_{d\to\infty}L_{j,d}\,d^{\frac{j+1}{2}}=0.
\]
Then the information complexity for the classes $C_d^k(L)$ satisfies 
\[
n(\eps,C_d^k(L)) 
 \,\le \, {\rm e}^j \, d^{j}
\]
for all $\eps\in(0,1)$ if $d\ge d(\eps)$ is large enough.
\end{prop}

\begin{proof}
Let the sequence $(x_d^*)_{d\in\N}$ and $R<\infty$ be 
provided by Property $(\Pbf)$. 
Assume first that $j\in\{2,3,\dots,k-1\}$. 
Then we take a Taylor polynomial of order $j$ 
at the point $x_d^*\in D_d$ which can be written as
$$  
T_{j}(x)=\sum_{\ell=0}^{j} \frac{f^{(\ell)}(x_d^*)(x-x_d^*)^\ell}{\ell!}
\ \ \ \ \mbox{for all}\ \ \ \ x\in D_d.
$$
Here we use the standard notation 
$A(x^\ell)=A(x,\dots,x)$ for the evaluation of an
$\ell$-linear map on the diagonal.
Recall that we consider here $f^{(\ell)}(x_d^*)$ as an $\ell$-linear map.
It is well-known that the error of the approximation of $f$ by $T_j$ can be written as
$$
f(x)-T_{j}(x)=(j+1)\,\int_0^1(1-t)^{j}\,
\frac{ f^{(j+1)} \big(x_d^*+t(x-x_d^*)\big)(x-x_d^*)^{j+1}}{(j+1)!} \dint t.
$$
For $x\in D_d$ with $\|x-x_d^*\|_2\le R\sqrt{d}$ we can now estimate the error as 
$$
\left| f(x)-T_{j}(x) \right| \le  \frac{1}{j!} \,\int_0^1(1-t)^{j}\, \dint t \, 
\norm{f^{(j+1)}} \, \norm{x-x_d^*}_2^{j+1} 
\le \frac{R^{j+1}}{(j+1)!}\, L_{j,d} \, d^{\frac{j+1}{2}}.
$$
If $j=k$, we use the same approximation and note that
$$
f(x)-T_{k}(x) = f(x)-T_{k-1}(x) -  \frac{f^{(k)}(x_d^*)(x-x_d^*)^k}{k!} 
= k \,\int_0^1(1-t)^{k-1}\, r_x(t)\, \dint t
$$
with
$$
r_x(t)=\frac{ f^{(k)} \big(x_d^*+t(x-x_d^*)\big)(x-x_d^*)^k-f^{(k)} 
\big(x_d^*\big)(x-x_d^*)^k}{k!}.
$$
Then 
for $x\in D_d$ with $\|x-x_d^*\|_2\le R\sqrt{d}$, we get
$$
 |r_x(t)| \le \frac{\Lip( f^{(k)}) \, t \, \norm{x-x_d^*}_2^{k+1}}{k!} 
        \le \frac{R^{k+1}}{k!} \, L_{k,d} \, d^{\frac{k+1}{2}}
\ \ \ \mbox{for}\ \  t\in[0,1],
$$
   and 
$$
 \left| f(x)-T_{k}(x) \right| \le \frac{R^{k+1}}{k!} \, L_{k,d} \, d^{\frac{k+1}{2}}.
$$
So for all $j\in\{2,3,\dots,k\}$ we have 
$$
 \left| f(x)-T_{j}(x) \right|  \le \frac{R^{j+1}}{j!} \, L_{j,d} \, d^{\frac{j+1}{2}}
$$
if $\|x-x_d^*\|_2\le R\sqrt{d}$.
For all such $j$, consider the algorithm
$$
Q_{j,d}(f)=\int_{D_d\cap B_R^d(x_d^*)} T_j(x)\,{\rm d}x
=\sum_{\beta:\,|\beta|\le j}\frac{D^\beta f(x_d^*)}{\beta!}\
\int_{D_d\cap B_R^d(x_d^*)}\prod_{j=1}^d\left(x_j-x_{d,j}^*\right)^{\beta_j}\,\dint x.
$$
Since
$\|f\|_\infty\le1$, we obtain
\[\begin{split}
\left|\int_{D_d}f(x)\,\dint x\ - \ Q_{j,d}(f)\right| 
\;&\le\; \left|\int_{D_d\cap B_R^d(x_d^*)}\left(f(x)-T_j(x)\right)
\,\dint x\right| \,+\, 
     \left|\int_{D_d\setminus B_R^d(x_d^*)}f(x)\,\dint x\right| \\
\;&\le\; \int_{D_d\cap B_R^d(x_d^*)}\left|f(x)-T_j(x)\right|\,\dint x \,+\, 
                                \lambda_d\bigl(D_d\setminus B_R^d(x_d^*)\bigr) \\
&\le\; \frac{R^{j+1}}{j!} \, L_{j,d} \, d^{\frac{j+1}{2}} \,+\, 
                                \lambda_d\bigl(D_d\setminus B_R^d(x_d^*)\bigr).
\end{split}\]
By the assumptions of the proposition, both terms tend to zero as $d$ goes to infinity. 
Thus, $Q_{j,d}$ yields an arbitrary small error if the dimension is large enough.

It remains to bound the cost of $Q_{j,d}$.
First note that $Q_{j,d}$ is not an admissible algorithm since
we can compute only function values. 
However, it is easy to see that we can approximate each 
partial derivative $D^\beta f(x^*)$
by divided differences with an arbitrary precision by 
computing only a number of function values that does not depend on the 
dimension $d$ (but on $|\beta|\le j\le k$). 
More precisely, we can compute $Q_{j,d}(f)$ for $f\in C_d^k(L)$, $j\le k$, 
up to an arbitrary error $\eta>0$ using 
$$
\binom{d+j}{j}\le\Biggl(\frac{{\rm e}(d+j)}{j}\Biggr)^j 
\le {\rm e}^j d^j
$$
(for $d,j\ge2$)
function values of $f$, see \cite{Vyb13}.
The proposition follows.
\end{proof}

\subsection{Partial Derivatives}

In this section we comment on results that can be deduced directly from 
the already proven statements. 
In particular, we state results for classes $\wt{C}_d^k(L)$ that are defined 
like $C_d^k(L)$ from~\eqref{eq:class}, but with conditions on 
arbitrary directional derivatives replaced by conditions only on 
\emph{partial} derivatives. The results follow solely by inclusion.

We define the function classes by
\[
\wt{C}_d^k(L) \,=\, \biggl\{f\in C^{k}(D_d)\ \ \big|\ \ \ \|f\|_\infty \le1,\ 
        \sup_{|\beta|=j}\Lip( D^\beta f) \le L_{j,d}\ \mbox{for all } j\le k\, \biggr\}.
\]
It is easy to see that for each $f\in C^k(D_d)$ and $j\le k$
\[\begin{split}
\sup_{|\beta|=j}\Lip( D^\beta f) 
\;&\le\; \sup_{\theta_1,\dots,\theta_j\in \sphere^{d-1}}
                                \Lip( D^{\theta_j} \dots D^{\theta_1} f)
\;=\; \sup_{\theta\in \sphere^{d-1}}
                                \Lip( D^{\theta} \dots D^{\theta} f)\\
\;&=\; \Lip( f^{(j)} ).
\end{split}\]

Let $\theta=(\theta_1,\dots,\theta_d)\in\sphere^{d-1}$ 
and $x,y\in D_d$. Noting that
\[
f^{(j)}(x)(\theta,\dots,\theta) \,=\,
\sum_{i_1,i_2,\dots,i_j=1}^d\frac{\partial^j f}{\partial x_{i_1}\,
\partial x_{i_2}\,\cdots\,\partial x_{i_j}}(x)\,
\theta_{i_1}\,\theta_{i_2}\,\cdots\,\theta_{i_j}
\]
we obtain

\[\begin{split}
\bigl|f^{(j)}(x)(\theta,\dots,\theta)-f^{(j)}(y)(\theta,\dots,\theta)\bigr| 
\;
&\le\; \biggl(\sum_{i=1}^d|\theta_i|\biggr)^j\; 
                                \sup_{|\beta|=j}\, \bigl|D^\beta f(x)-D^\beta f(y)\bigr| \\
&\le\; d^{j/2}\; 
                                \sup_{|\beta|=j}\, \bigl|D^\beta f(x)-D^\beta f(y)\bigr|.
\end{split}\]
This implies 
$\Lip(f^{(j)}) \;\le\; d^{j/2}\,\sup_{|\beta|=j}\,\Lip(D^\beta f)$
and thus
\begin{equation}\label{dirpar}
C_d^k(L) \,\subset\, \wt{C}_d^k(L) 
\,\subset\, C_d^k\bigl((d^{j/2} L_{j,d})_{j,d\in\N}\bigr)
\end{equation}
for arbitrary double sequences $L=(L_{j,d})_{j,d\in\N}$. 
An immediate consequence is that all previously proven 
lower bounds on the information complexity for $C_d^k(L)$ also hold 
unchanged 
for $\wt{C}_d^k(L)$.
Moreover, we obtain from Theorem~\ref{thm4}
the following proposition.

\begin{prop} \label{prop:partial}
Let $(D_d)_{d\in\N}$ be cubes $(0,1)^d$
or convex sets of small radius in the sense of~\eqref{smallradius}. 
Then, for all $k\in\N$, the conditions
\[
\limsup_{d\to\infty}L_{0,d} \, \sqrt{d} > 0 \qquad \mbox{and} \qquad 
\limsup_{d\to\infty}L_{j,d} \, d > 0 \ \mbox{ for }\ j=1,\dots,k
\]
imply the curse of dimensionality for $\wt{C}_d^k(L)$.
If there exists $j\in\{0,1,\dots,k\}$ with
$$ \lim_{d\to\infty}L_{j,d} \, d^{j+\frac{1}{2}} = 0 $$
then the curse does not hold.
\end{prop}

\section{Functions with Infinite Smoothness}  
\label{sec:Functions with Infinite Smoothness}

In this section we deal with $C^\infty$ functions. 
The classes we consider are now of the form  
$$ C_d^{\infty}(L) =\{f\in C^{\infty}(D_d)\ \ \big|\ \ \ \|f\|_\infty\le1,\ 
        \Lip( f^{(j)}) \le L_{j,d}\ \mbox{for } j=0,1,2\dots\ \}.$$

The main result of this section is 

\begin{thm} \label{thm5}
Let $(D_d)_{d\in\N}$ be the cubes $(0,1)^d$
or 
sets of small radius in the sense of \eqref{smallradius}. 
Then the conditions
\[
\limsup_{d\to\infty}L_{0,d} \, \sqrt{d} > 0 \qquad \mbox{and} \qquad 
\limsup_{d\to\infty}L_{j,d} \, d > c\, (j!)^{1+\eta} \ \mbox{ for }\ j\in\N,
\]
where $c,\eta>0$,
imply the curse of dimensionality for $C_d^\infty(L)$.

On the other hand, if $D_d$ is convex and 
there exists $j\in\N$ with
$$ \lim_{d\to\infty}L_{j,d} \, d^{\frac{j+1}{2}} = 0 $$
then the curse does not hold.
Furthermore, if there exist constants $c<\infty$ and $\delta>0$ such that 
\[
L_{j,d} \,\le\, c\, (2-\delta)^{j}\, j!\, \,d^{-\frac{j+1}{2}} 
\]
for all $j,d\in\N$, then the problem of numerical integration for the 
classes $C_d^\infty(L)$ is 
quasi-polynomially tractable, i.e., 
\[
\ln\bigl(n(\eps,C_d^\infty(L))\bigr) \,\le\, C (1-\ln\eps)(1+\ln d)
\]
for some absolute $C<\infty$.
\end{thm}

\subsection{Lower Bounds}    \label{subsec:Lower Bounds 4}

To prove lower bounds for the classes $C_d^\infty(L)$ we 
basically 
use the same fooling functions 
$(f_k)$ as in Section~\ref{subsec:Lower Bounds 3}, but with a 
different sequence $(\alpha_j)$. 
Moreover, we need to
take the limit for $k\to\infty$. 
For this reason, we first study the convolution of infinitely many 
indicator functions $g_j$. The resulting function in the one-dimensional case is reminiscent 
of the up-function of Rvachev, see \cite{Rv90}.

Recall the definition \eqref{eq:gk} of the $L_1$-normalized indicator
functions $g_j$ of the ball of radius $\alpha_j \delta \sqrt{d}$
for $j\in\natural$.
Now we define 
$$ 
 G_k = g_1 \ast \ldots \ast g_k.
$$  
Observe that $G_2$ is Lipschitz and $G_3 \in C^1(\R^d)$.
Theorem \ref{thm:conv} implies that $G_k$ is also Lipschitz
for $k\ge 2$ with $\Lip(G_k)\le \Lip(G_2)$ and, more generally, $G_{k}\in C^{\ell}(\R^d)$ 
for $k > \ell +1 \ge 1$ with 
$$
 \Lip(G_k^{(\ell)}) \le \Lip(G_{\ell+2}^{(\ell)}).               
$$ 
This implies that the limit function
$$
G_\infty = \lim_{k\to \infty} G_k
$$
exists, the convergence is uniform, $G_\infty \in C^\infty(\R^d)$ and
$$
 G^{(\ell)}_\infty = \lim_{k\to \infty} G^{(\ell)}_k
$$
uniformly for all $\ell\ge 0$.
Indeed, fix a direction $\theta\in\sphere^{d-1}$ and let $L_\ell$ 
be the operator
of $\ell$-times differentiation in direction $\theta$. Then
$$
 L_\ell G_n = L_\ell (G_{\ell+3}) \ast g_{\ell+4} \dots \ast g_{n}
$$
for $n\ge \ell+3$ and the $L_1$-normalization of $g_{n+1}$ imply
\begin{eqnarray*}
|L_\ell G_{n+1}(x) - L_\ell G_n(x) | &=&
\left| \int_{\R^d}   \big(L_\ell G_n(x-y) - L_\ell G_n(x) \big) 
\, g_{n+1}(y)  {\rm d} y  \right|  \\
&\le& \int_{B_{n+1}}   \big|L_\ell G_n(x-y) 
- L_\ell G_n(x) \big| \, g_{n+1}(y)  {\rm d} y \\
&\le& \Lip(L_\ell G_n) \alpha_{n+1} \delta \sqrt{d} 
\le \Lip(G_{\ell+2}^{(\ell)})  \delta \sqrt{d}\, \alpha_{n+1}. \\
\end{eqnarray*} 
This leads to
$$
|L_\ell G_{n+m}(x) - L_\ell G_n(x) | 
\le \Lip(G_{\ell+2}^{(\ell)})  \delta \sqrt{d} \sum_{k=n+1}^\infty \alpha_{k} 
$$
for all $m\ge 1$. 
Now the summability of the sequence $(\alpha_j)$ shows that  $ \big(L_\ell G_n\big)$ is a
uniform Cauchy sequence proving the claim.

Now we can define our final fooling function 
$$ 
 f_\infty = f \ast G_\infty
$$
using again the initial function $f$ constructed in Section \ref{subsec:Lower Bounds 2}.
Then the functions 
$$ 
 f_k = f \ast G_k
$$
converge uniformly to $f_\infty$. 
We also have uniform convergence of the corresponding derivatives.
By induction, we obtain the following properties 
from Theorem \ref{thm:conv}:
\begin{itemize}
        \item $f_\infty\in C^{\infty}$, 
        \item $f_\infty(x)=0$ for $x\in K$,
        \item $f_\infty(x)=1$ for $x\notin K_{3\delta}$,
        \item $\Lip(f_\infty) \le \frac{2}{\delta \sqrt{d}}$ and 
        \item $L^*_{j,d} := \Lip(f_\infty^{(j)}) \le \frac{40}{\delta^2 d} 
                \biggl(\prod_{i=1}^{j-1} \frac{1}{\delta \alpha_i}\biggr) $, 
                                $j=1,\dots,\ell+1$.
\end{itemize}
Using the sequence $\alpha_j=c_\eta j^{-1-\eta}$ for $\eta>0$ 
with $c_\eta=\zeta(1+\eta)^{-1}$
the last estimate yields 
\begin{itemize}
\item $L^*_{j,d} := \Lip(f_\infty^{(j)}) 
\le \frac{40}{d} \delta^{-1-j} c_\eta^{1-j} \big((j-1)!\big)^{1+\eta} $, 
                                $j=1,\dots,\ell+1$.
\end{itemize}

By the third property of $f_\ell$ we additionally obtain
\[
\int_{D_d} f_\infty(x) \,\dint x 
\,\ge\, \int_{D_d\setminus K_{3\delta}} f_\ell(x) \,\dint x 
\,=\, 1- \lambda_d(K_{3\delta}\cap D_d).
\]
The proof of the lower bound in Theorem \ref{thm5} is then finished exactly as in the
case of finite smoothness $k$ in Subsection \ref{subsec:Lower Bounds 3}.
Note that in this case, we cannot use the argument of scaling
for the class $C^\infty_d(L)$. 
Thus, the $(j!)^{1+\eta}$ remains in the asymptotic lower bound 
for $(L_{j,d})$. 

\subsection{Upper Bounds}    \label{subsec:Upper Bounds 4}

We prove upper bounds on the information complexity for the classes
$C_d^\infty(L)$ of infinitely differentiable functions.
First of all, note that Proposition~\ref{prop:ub_higher} 
holds unchanged for $k=\infty$. This is summarized
in the following proposition.

\begin{prop} \label{prop:ub_infinite}
Let $(D_d)_{d\in\N}$ be a sequence of convex sets that satisfies $(\Pbf)$.
Additionally assume that there exists $j\in\N$ such that
\[
\lim_{d\to\infty}L_{j,d}\,d^{\frac{j+1}{2}}=0.
\]
Then the information complexity for the classes $C_d^\infty(L)$ satisfies 
\[
n(\eps,C_d^\infty(L)) 
 \,\le \, {\rm e}^j \, d^{j}
\]
for all $\eps\in(0,1)$ if $d\ge d(\eps)$ is large enough.
\end{prop}

Again, this result only shows that, under the given assumptions, 
the curse of dimensionality does not hold for $C_d^\infty(L)$.
The next
proposition improves this result in the sense that we 
obtain \emph{quasi-polynomial tractability}, that is 
\[
\ln\bigl(n(\eps,C_d^\infty(L))\bigr) \,\le\, C (1-\ln\eps)(1+\ln d)
\ \ \ \mbox{for all}\ \ \eps\in(0,1),\ d\in\natural,
\]
for some absolute $C<\infty$, if the asymptotic conditions in 
the last proposition are replaced by uniform bounds. 
We prove two different variants of this
result which differ by 
the power of the $j!$ in the required bounds.

\begin{prop} \label{prop:ub_infinity2}
Let $(D_d)_{d\in\N}$ be a sequence of convex sets with $\lambda(D_d)=1$ and define 
$R_d = \rad(D_d)/\sqrt{d}$.
Additionally assume that there exist $a>1$ and $c>0$ such that
\[
L_{j,d} \,\le\, c\, a^{-j}\, j!\,  R_d^{-j-1} \,d^{-\frac{j+1}{2}} 
\]
for all $j,d\in\N$.
Then the problem of numerical integration for the 
classes $C_d^\infty(L)$ is 
quasi-polynomially tractable.
\end{prop}

This proves the respective part of Theorem~\ref{thm5} 
since the cube and convex sets of 
small radius satisfy $R_d\le 1/2$, 
and then $R_d/a\ge2/a=2-\delta$ for some positive $\delta$.

For $D_d=(0,1)^d$ a similar result is contained in \cite{Vyb13}.
The paper~\cite{Vyb13} also studies
other classes of $C^\infty$ functions defined on the ball 
and upper bounds are obtained using Taylor approximations.

\begin{proof}
The proof is completely analogous to the proof of 
Proposition~\ref{prop:ub_higher}. Recall the definitions from there.
In particular, we obtain in the same fashion 
for all $j\in\N$ and $x\in D_d$ that
\begin{equation} \label{eq:error_taylor}
 \left| f(x)-T_{j}(x) \right|  
\,\le\, \frac{R_d^{j+1}}{j!} \, L_{j,d} \, d^{\frac{j+1}{2}}
\,\le\, c\, a^{-j}
\end{equation}
which is smaller than $\eps>0$ if 
$j\ge k_\eps := \lceil\log_a(\frac{c}\eps)\rceil$. 
Thus,
\[
\left|\int_{D_d}f(x)\,\dint x\ - \ Q_{k_\eps,d}(f)\right| 
\,\le\, \eps.
\]

Again, it remains to bound the cost of $Q_{k_\eps,d}$.
Exactly as in the proof of Proposition~\ref{prop:ub_higher}, 
see also \cite{Vyb13}, 
we can compute $Q_{k_\eps,d}(f)$ for $f\in C_d^\infty(L)$, 
up to an arbitrary error $\eta>0$ using 
$$
\binom{d+k_\eps}{k_\eps}\le\Biggl(\frac{{\rm e}(d+k_\eps)}{k_\eps}\Biggr)^{k_\eps} 
\le {\rm e}^{k_\eps} d^{k_\eps}
$$
function values. 
We conclude
\[
\ln\bigl(n(\eps,C_d^\infty(L))\bigr) 
\,\le\, k_\eps (1+\ln d)
\,\le\, (1+\ln c) (1+1/\ln a) (1-\ln\eps)(1+\ln d).
\]
The proposition follows.
\end{proof}

The last proposition of this section deals with the case where the (uniform) 
upper bounds on $(L_{j,d})$ have a smaller power of $j!$.
This allows us
to conclude weak tractability for sequences $(D_d)$ with 
slightly larger radii, i.e.,~$\rad(D_d)\prec d^{\,1-\eps}$, $\eps>0$, 
while the order of $d$ in $(L_{j,d})$ remains the same.
Recall that numerical 
integration is called \emph{weakly tractable} for 
$C^k_d(L)$ iff
$$
\lim_{d+\eps^{-1}\to \infty}\frac{\ln\, n(\eps,C^k_d(L))}
{d+\eps^{-1}}=0.
$$
This means that $n(\eps,C^k_d(L))$ is not exponential
in $d$ and in $\eps^{-1}$ but can be exponential in $d^{\,\alpha}$ or 
$\eps^{-\alpha}$ for $\alpha\in(0,1)$. 
Obviously, the curse of dimensionality implies
that weak tractability does not hold. However, the converse statement
is not necessarily true, see e.g., \cite{NW08,NW10} for more details.

\begin{prop} \label{prop:ub_infinity3}
Let $(D_d)_{d\in\N}$ be a sequence of convex sets with $\lambda(D_d)=1$ and define 
$R_d = \rad(D_d)/\sqrt{d}$.
Additionally assume that 
\[
L_{j,d} \,\le\, c\, a^{-j} (j!)^{1-\eta}\, \,d^{-\frac{j+1}{2}} 
\]
for all $j,d\in\N$ and some $a,c,\eta>0$. 
Then the information complexity for the classes $C_d^\infty(L)$ satisfies 
\[
\ln\bigl(n(\eps,C_d^\infty(L)) \bigr)
 \,\le\,  \Bigl(1+\eu^{1+1/\eta}(R_d/a)^{2/\eta} \Bigr) 
                        \Bigl(1+\ln\frac{c}{\eps}\Bigr) \bigl(1+\ln d \bigr)
\]
for all $\eps\in(0,1)$.

In particular, the problem of numerical integration for the 
classes $C_d^\infty(L)$ is 
quasi-polynomially tractable for
$R_d\le R<\infty$, and 
weakly tractable for $R_d\le R d^{1/2-\delta}$, $R<\infty$, $\delta>0$, 
and $\eta>1-2\delta$.
\end{prop}


Since the proof is again analogous to the proof of 
Proposition~\ref{prop:ub_higher} and 
the proof of 
Proposition~\ref{prop:ub_infinity2}, we omit the details.

\subsection{The case of $L_{j,d}=1$}
The same proof technique can be used to obtain weak tractability 
for the unit ball of our function classes $C_d^\infty$ with $L_{j,d}=1$
if the radii of the sets $D_d$ are not too large. 
That is, for the classes
\[
C_d^{\infty}(1) \,=\,\bigl\{f\in C^{\infty}(D_d)\ \ \big|\ \ \ \|f^{(k)}\| \le1,\ 
        \ \mbox{for all } k\in\N\, \bigr\}. 
\]

\begin{cor} \label{cor:ub_infinity_unit_ball}
Let $(D_d)_{d\in\N}$ be a sequence of convex sets with $\lambda(D_d)=1$.
Then the information complexity for the classes $C_d^\infty(1)$ defined above satisfies 
\[
\ln\bigl(n(\eps,C_d^\infty(1)) \bigr)
 \,\le\,  (1+\ln d) \max\Bigl\{\eu^2\, \rad(D_d),\, \ln\bigl(\eps^{-1}\, \rad(D_d)\bigr)\Bigr\}
\]
for all $\eps\in(0,1)$ and $d\in\N$.

Hence, weak tractability holds if
$$
\lim_{d\to\infty}\frac{(1+\ln\,d)\,\rad(D_d)}{d}=0.
$$
In particular, weak tractability holds for the unit cubes $D_d=(0,1)^d$
and $\ell^d_p$-balls with $p\in(1,\infty]$.
\end{cor}

\begin{proof}
Note that we know from \eqref{eq:error_taylor} 
and from the assumed condition $L_{j,d}=1$ 
that
for all $x\in D_d$ we have 
\[
\left| f(x)-T_{j}(x) \right|  
\,\le\, \frac{\rad(D_d)^{j+1}}{j!}
\,\le\, \rad(D_d)\,\left(\frac{\rad(D_d)\,\eu}{j}\right)^j,
\]
where the last estimate holds due to Stirling's formula.  
The right hand side of this inequality is clearly smaller than $\eps$ 
if 
$$
j\ge k_\e:=\left\lceil \max\bigl\{\eu^2\, \rad(D_d),\, \ln(\eps^{-1}\, \rad(D_d))\bigr\}\right\rceil.
$$ 
Then to get the bound on $\ln\bigl(n(\eps,C_d^\infty(1)) \bigr)$ we 
proceed as the proof of Proposition~\ref{prop:ub_infinity2}.
The rest is easy since the radius of the unit cube $[0,1]^d$
is $\tfrac12\sqrt{d}$ and the radius of $\ell^d_p$-ball is $o(d)$,
which follows from \eqref{eq:radlp} and Stirling's approximation. 
\end{proof}

It is interesting to notice that the last result is too weak to establish
weak tractability for the corresponding class $\tilde C_d^\infty(1)$ 
of all partial derivatives bounded by one. Indeed, using~\eqref{dirpar} with $k=\infty$ 
we need to consider $C^\infty_d(L)$ with $L_{j,d}= d^{j/2}$. Then~\eqref{eq:error_taylor} 
implies that for all $x\in D_d$ we have 
$$
\left| f(x)-T_{j}(x) \right|  
\,\le\, \frac{\rad(D_d)^{j+1}\,d^{j/2}}{j!}
\,\le\, \rad(D_d)\,\left(\frac{\rad(D_d)\,\eu\,\sqrt{d}}{j}\right)^j.
$$
Then the integration error is at most $\e$ if 
$$
j\ge k_\e:=\left\lceil
\max\bigl\{\eu^2\, \rad(D_d)\,\sqrt{d},\, \ln(\eps^{-1}\, \rad(D_d))\bigr\}\right\rceil.
$$
This corresponds to the estimate
$$
\ln\bigl(n(\eps,\tilde C_d^\infty(1)) \bigr)
 \,\le\,  (1+\ln d) \max\Bigl\{\eu^2\, \rad(D_d)\,\sqrt{d}, \ln\bigl(\eps^{-1}\, \rad(D_d)\bigr)\Bigr\},
$$
which is too weak to show weak tractability. Therefore weak tractability for the class
$\tilde C_d^\infty(1)$ of all partial derivatives bounded by one remains an open problem. 

We return to the class $C_d^\infty(1)$ 
of all directional derivatives bounded by one. 
We know that weak tractability holds for this class. Can we say something more
on different notions of tractability for 
$C_d^\infty(1)$?  Yes, we can check that strong polynomial tractability does not hold,
i.e., $n(\eps,C_d^\infty(1))$ cannot be bounded by a polynomial in $\e^{-1}$ independently of $d$.
This follows from~\cite{Wo03} who proved that strong tractability does not hold for the larger class
$\tilde C_d^\infty(1)$ but his proof also applies to the class $C_d^\infty(1)$.
Unfortunately, it is all what we can say. In particular, quasi-polynomial tractability is open
for the class $C_d^\infty(1)$.

\section{Weak and Uniform Weak Tractability}

In this section we study the notions of weak and uniform weak tractability and 
show that the problem is \emph{not} uniformly weakly tractable 
as long as the bounds $L_{j,d}$ on the Lipschitz constants of the 
derivatives decay slower than \emph{any} inverse polynomial in $d$. 
We prove this result without any additional condition on the 
sequence of domains $(D_d)_{d\in\N}$ besides $\lambda_d(D_d)=1$.
This is surprising since we already proved polynomial (in $d$) 
upper bounds on the information complexity if one $L_{j,d}$ decays (roughly) 
faster than $d^{-(j+1)/2}$ and the dimension is large enough 
depending on $\eps$, see Propositions~\ref{prop:ub_lip}, 
\ref{prop:ub_lipgrad_P} and \ref{prop:ub_higher}.

The concept of weak tractability was recently strengthened 
in~\cite{Sied12} by introducing the notion of
\emph{uniform weak tractability}, which holds for multivariate integration
defined over the class $C^k_d(L)$ iff
$$
\lim_{d+\eps^{-1}\to \infty}\frac{\ln\, n(\eps,C^k_d(L))}
{d^\alpha+\eps^{-\alpha}}=0 \ \ \ \ \mbox{for all}\ \ \ \ \ \alpha\in(0,1).
$$
For uniformly weak tractability,
$n(\eps,C^k_d(L))$ is \emph{not exponential}
in $d^{\,\alpha}$ and $\eps^{-\alpha}$ for all $\alpha\in(0,1)$.

Now assume that the double sequence $L=(L_{j,d})_{j,d\in\N}$ satisfies
\[
\limsup_{d\to\infty}L_{j,d} \, d^{\,m} > 0
\]
for all $j\in\{0,1,\dots,k\}$ and some $m<\infty$.
This means 
that there exists a constant $c>0$ for which
$L_{j,d}\ge c d^{-m}$ for infinitely many $d$.

We prove that under this condition the problem of numerical integration 
for the class $\widebar{C}_d^k(L)$, 
see~Section~\ref{subsec:Function Classes},
is not uniformly weakly tractable, 
independent from the sequence $(D_d)_{d\in\N}$. 
This also implies that the problem for $C_d^k(L)$ is not uniformly weakly 
tractable.
For this recall the definition of the class 
$F_{d,k,\delta}$ from~\cite[Sec.~4]{HNUW12}. 
It follows from the definition of $F_{d,k,\delta}$ that
\[
\wt{F}_{d,k} :=\left\{f|_{D_d} \ \big|\ f\in F_{d,k,\delta}\ \text{ for }\, 
\delta^{-1}=2\sqrt{18\eu\pi}\right\} 
\subset \widebar{C}_d^k(d^m L)
\]
for infinitely many $d\in\N$.
This leads to
\[
n(\eps d^{-m},C^k_d(L)) \,=\, n(\eps,d^m C^k_d(L)) 
\,\ge\, n(\eps,C^k_d(d^m L)) \,\ge\, n(\eps,\widebar{C}^k_d(d^m L))
\,\ge\, n(\eps,\wt{F}_{d,k})
\]
for all $\eps\in(0,1)$ and infinitely many $d$. 
We know from~\cite[Thm.~2]{HNUW12} that the right hand side of this inequality 
is bounded from below by $(1-\eps) 2^d$, 
see also ~\cite[Remark~1]{HNUW12}.
We now choose
a sequence $(d_i,\eps_i)_{i\in\N}=(d_i,\frac{1}{2}d_i^{-m})_{i\in\N}$, 
such that $\lim_{i}d_i=\infty$ and 
$n(\eps_i,C^k_{d_i}(L))\ge 2^{d_i-1}$ for all $i\in\N$.
The
limit in the definition of uniform weak tractability is then lower
bounded by
\[
\limsup_{d\to\infty}\frac{\ln\, n(\frac{1}{2}d^{-m}, C^k_d(L))}
        {d^\alpha+\bigl(\frac{1}{2}d^{-m}\bigr)^{-\alpha}} 
\,\ge\, \lim_{i\to\infty}\frac{(d_i-1) \ln2}
        {d_i^\alpha+2^\alpha d_i^{\alpha m}} 
\,>\, 0
\]
if we take $\alpha\le m^{-1}$. 
This contradicts uniform weak tractability for the classes 
$C^k_{d}(L)$ for all finite $k$.

We now 
turn to the case $k=\infty$. 
We prove 
uniform weak tractability for the classes $C_d^\infty(L)$ if 
\[
L_{j,d} \,\le\, c_d\;j!\, a_d^{-j}\, R_d^{-j-1} \,d^{-\frac{j+1}{2}} 
\]
for all $j,d\in\N$ and some $a_d,c_d>1$ that may depend on $d$, 
where $R_d = \rad(D_d)/\sqrt{d}$. 
The upper bound on the information complexity from
Proposition~\ref{prop:ub_infinity2} (see the last inequality 
in the proof) implies that
\[
\ln\, n(\eps,C^\infty_d(L)) 
\,\le\, (1+\ln c_d) (1+1/\ln a_d) (1-\ln\eps)(1+\ln d).
\]

Plugging this into the definitions of weak and uniform weak 
tractability 
we obtain

\begin{thm} \label{prop:uwt}
Let $(D_d)_{d\in\N}$ be a sequence of convex sets with $\lambda(D_d)=1$ and define  
$R_d = \rad(D_d)/\sqrt{d}$.
Additionally assume that 
\[
L_{j,d} \,\le\, c_d\;j!\, a_d^{-j}\, R_d^{-j-1} \,d^{-\frac{j+1}{2}} 
\ \ \ 
\mbox{for all}\ \  j,d\in\N. 
\]
Then 
\begin{itemize}
\item
the problem of numerical integration for the 
classes $C_d^\infty(L)$ is 
uniformly weakly tractable if
\[
c_d \,\le\, d^m \quad\text{ and }\quad a_d \,\ge\, 1+1/(1+\ln d)^m
\]
for some $m<\infty$, 
\item
the problem of numerical integration for the 
classes $C_d^\infty(L)$ is 
weakly tractable if
\[
c_d \,\le\, \exp( d^{\,b_1}) \quad\text{ and }\quad 
a_d \,\ge\, 1+1/d^{\,b_2}
\]
for some $b_1,b_2\in[0,1)$ with $b_1+b_2<1$. 
\end{itemize}
\end{thm}

Theorem~\ref{prop:uwt} states, in particular, that
weak tractability of numerical integration holds 
for the unit cube $D_d=(0,1)^d$
and for the classes $C_d^\infty(L)$ 
if we take $b_1=\tfrac12$, $b_2<\tfrac12$, $c_d=\exp(\sqrt{d})$, $a_d=1+d^{-b_2}$, $R_d=\tfrac12$,
and with bound 
\[
L_{j,d} \,\le\, 2\,j!\, \left(\frac2{1+d^{-b_2}}\right)^j \,d^{-\frac{j+1}{2}}\, \eu^{\sqrt{d}}
\]
for all $j,d\in\N$. It can be checked that we cannot take $L_{j,d}=\mbox{constant}>0$ to satisfy the last
inequality. This shows that the conditions on $L_{j,d}$ in Corollary~\ref{cor:ub_infinity_unit_ball}
and Theorem~\ref{prop:uwt} are different.

\bigskip

\noindent
{\bf Acknowledgement.}
We thank Winfried Sickel for valuable comments and the reference~\cite{Rv90}.
We thank Jan Vyb\'\i ral for valuable comments and an early version 
of his paper \cite{Vyb13}.

\end{document}